\documentclass{amsart}
\usepackage[utf8]{inputenc}
\usepackage[dvipsnames]{xcolor}
\usepackage{geometry,amsmath,amsthm,amssymb,graphics,graphicx,tikz,tikz-cd,mathrsfs,fancyhdr,listings,wrapfig,array,float}
\usepackage{adjustbox}
\usepackage[shortlabels]{enumitem}
\usepackage{hyperref}
\hypersetup{colorlinks=true,citecolor=purple,linkcolor=purple}

\usepackage{style}
\usepackage{comment}

\newcommand{\x}{\textbf{x}}
\renewcommand{\a}{\textbf{a}}
\renewcommand{\b}{\textbf{b}}

\DeclareMathOperator{\wt}{wt}

\title{Symmetric Group Action of the Birational R-matrix}
\author{Sunita Chepuri, Feiyang Lin}
\date{August 7, 2020}

\begin{document}

\begin{abstract}
The birational $R$-matrix is a transformation that appears in the theory of geometric crystals, the study of total positivity in loop groups, and discrete dynamical systems. This $R$-matrix gives rise to an action of the symmetric group $S_m$ on an $m$-tuple of vectors.  While the birational $R$-matrix is precisely the formula corresponding to the action of the simple transposition $s_i$, explicit formulas for the action of other permutations are generally not known.  One particular case was studied by Lam and Pylyavskyy as it relates to energy functions of crystals. In this paper, we will discuss formulas for several additional cases, including transpositions, and provide combinatorial interpretations for the functions that appear in our work.
\end{abstract}

\maketitle

\section{Introduction}

The study of total positivity began in the 1930's with the discoveries of Schoenberg~\cite{S1930}, regarding variation-diminishing properties of the totally nonnegative part of $GL_n(\R)$, and Gantmacher--Krein~\cite{GK1937}, regarding spectral properties of the totally positive part of $GL_n(\R)$.  Since then, totally positive and totally nonnegative matrices have been found to have applications in many areas of math and physics.

One of the most important classical results in total positivity is the Loewner--Whitney Theorem~\cite{L1955,W1952}.  This theorem gives a set of generators, with easily computable relations, for the totally nonnegative part of $\GL_n(\R)$.  Lusztig~\cite{L1994} revolutionized the field by using the Loewner--Whitney Theorem to generalize the concept of total nonnegativity in $GL_n(\R)$ to other Lie groups.

In~\cite{LP2012}, Lam and Pylyavskyy explored total positivity in setting of loop groups.  One of their main results was an analogue of the Loewner--Whitney Theorem for the upper unitriangular part of the formal loop group. The relations between generators in this setting led to their definition of the birational $R$-matrix, a transformation on an ordered pair of vectors in $\R_{>0}$.

The generators found by Lam and Pylyavskyy in~\cite{LP2012} correspond to cylindric networks via a boundary measurement map. The birational $R$-matrix, which describes relations between the generators, can also be interpreted in terms of cylindric networks: it describes a semi-local move on cylindric networks that preserves boundary measurements. The connection between the birational $R$-matrix and cylindric networks is explored more fully in Part 2 of~\cite{LP2013}. This work has since been extended by the first author to the context of plabic networks~\cite{C2020}.

The birational R-matrix is also related to several other areas of mathematics.  This transformation plays an important role in the study of geometric crystals~\cite{BK2010, E2003}.  It tropicalizes to the combinatorial R-matrix~\cite{KKMMNN1992} and can be obtained using cluster algebras via the cluster $R$-matrix of Inoue--Lam--Pylyavskyy~\cite{ILP2016}. In addition, it has applications to discrete Painlev{\'e} dynamical systems~\cite{KNY2002} and box-ball systems~\cite{LPS2014}.

We now proceed to define the birational R-matrix. Given $\a = (a_1, \dots, a_n),\b = (b_1, \dots, b_n)\in\R^n_{>0}$, let \[\kappa_i(\a, \b) = 
\sum_{j = i}^{i+n-1}
\prod_{k = i+1}^{j}b_k 
\prod_{k = j+1}^{i+n-1}a_k,\] where the indices $k$ are taken modulo $n$.
Then we can define a map
\[\eta: (\a, \b) \mapsto (\b', \a')\]
where $\a' = (a_1', \dots, a_n'), \b' = (b_1', \dots, b_n')$, and
\[a_i' = \frac{a_{i-1}\kappa_{i-1}(\a, \b)}{\kappa_{i}(\a, \b)}, \ \ b_i' = \frac{b_{i+1}\kappa_{i+1}(\a, \b)}{\kappa_{i}(\a, \b)}.\]

\begin{example}
For $n = 4$, 
\[a_2' = a_1 \frac{\kappa_1(\a, \b)}{\kappa_2(\a, \b)}= a_1 \frac{a_2a_3a_4+b_2a_3a_4+b_2b_3a_4+b_2b_3b_4}{a_3a_4a_1+b_3a_4a_1+b_3b_4a_1+b_3b_4b_1}.\]
\end{example}

Let $\x_1,\dots ,\x_m\in\R^n_{>0}$ where $\x_i=(x_i^{(1)},\dots,x_i^{(n)})$ and upper indices are considered modulo $n$.  For $1\leq i<m$, we define \[\eta_i(\x_1, \dots, \x_m) = (\x_1, \dots, \x_{i-1}, \eta(\x_i, \x_{i+1}), \x_{i+2}, \dots, \x_m).\]

\begin{theorem}[\cite{LP2012} Lemma 6.1, Theorem 6.3]
The birational $R$-matrix has the following properties:
\begin{itemize}
    \item $\eta$ is an involution: for $1 \leq i < m$, $\eta_i^2 = 1$;
    \item $\eta$ satisfies the braid relation: for $1 \leq i < m-1$,
    \[\eta_i\eta_{i+1}\eta_i(\x_1, \dots, \x_m) = \eta_{i+1}\eta_i\eta_{i+1}(\x_1, \dots, \x_m).\]
\end{itemize}
\end{theorem}
This implies that the birational $R$-matrix defines an action of the symmetric group $S_m$ on $\x_1,\dots,\x_m$.  Let $s_i$ denote the transposition that switches $i$ and $i+1$. Then we obtain a symmetric group action by defining 
\[s_i(\x_1, \dots, \x_m) = \eta_i(\x_1, \dots, \x_m).\] To refer to specific variables after applying a permutation, we write $s(\x_1, \dots, \x_m) = (s(\x_1), \dots, s(\x_m))$ where $s(\x_i) = (s(x_i^{(1)}), \dots, s(x_i^{(n)}))$.

\vspace{0.1in} \textbf{Main Problem.} For any $s \in S_m$, $1 \leq i \leq m$ and $1 \leq r \leq n$, how can we write $s(x_i^{(r)})$ explicitly as a rational function in the original variables?
\vspace{0.1in}

In~\cite{BFZ1996}, Berenstein, Fomin, and Zelevinsky asked a similar question: given a minimal factorization of a totally positive matrix, how can we explicitly write the parameters of another minimal factorization?  This line of study led to the development of cluster algebras. Our main problem is the loop group analogue of this question.

This paper will proceed as follows.  In Section~\ref{sec:oneshifts} of this paper, we build on results in~\cite{LP2010} to produce explicit formulas for the action of $s_{j-1}s_{j-2}\dots s_i$ and $s_{i}s_{i+1}\dots s_{j-1}$ when $1\leq i< j<m$. These formulas are given in terms of functions we denote as $\tau$, $\sigma$, and $\bar{\sigma}$.  In Section \ref{sec:trans}, we introduce $\Omega$ functions and we state formulas for the action of permutations $s_ks_{k+1}\dots s_{j-2}s_{j-1}s_{j-2}\dots s_i$ where $1\leq i \leq k <j<m$ and $s_{k-1}s_{k-2}\dots s_{i+1}s_is_{i+1}\dots s_{j-1}$ where $1\leq i < k \leq j<m$.  Note that when $k=i$ in the first case this permutation is the transposition of $i$ and $j$, and similarly when $k=j$ in the second case.  Section~\ref{sec:transProof} contains the proof of a technical lemma needed in Section~\ref{sec:trans}.  Lastly, in Section~\ref{sec:comb-interp}, we provide combinatorial interpretations of the $\tau$, $\sigma$, $\bar{\sigma}$ and $\Omega$ functions.

\section{Formulas}

Our formulas rely heavily on functions we denote as $\tau, \sigma$ and $\bar{\sigma}$.  The $\tau$ and $\sigma$ functions were defined by Lam and Pylyavskyy in~\cite{LP2010}.  The $\bar{\sigma}$ function is dual to the $\sigma$ function.

Let $n$ be a positive integer, $k$ a nonnegative integer, and let $1 \leq r \leq n$. Then $\tau_k^{(r)}$ is defined as follows:
\[\tau_k^{(r)}(\x_1,\x_2,\dots,\x_m)=\sum_{1\leq i_i\leq i_2\leq\dots\leq i_k\leq m}x_{i_1}^{(r)}x_{i_2}^{(r-1)}\dots x_{i_k}^{(r-k+1)}\]
where no index appears more than $n-1$ times in the sum.  By convention, $\tau_0^{(r)}(\x_1,\x_2,\dots,\x_m)=1$ and $\tau_k^{(r)}(\x_1,\x_2,\dots,\x_m)=0$ if $k$ is negative or if $k>m(n-1)$.

\begin{example}
Let $n = 4$. Then $\tau_{5}^{(3)}(\x_1,\x_2)=x_1^{(3)}x_1^{(2)}x_1^{(1)}x_2^{(4)}x_2^{(3)}+x_1^{(3)}x_1^{(2)}x_2^{(1)}x_2^{(4)}x_2^{(3)}$.
\end{example}

The $\sigma$ and $\bar{\sigma}$ functions are defined using $\tau$. We can think of them as the $\tau$ functions with the caveat that $\x_1$ or $\x_m$ variables are now allowed to appear more than $n-1$ times.
\[\sigma_k^{(r)}(\x_1,\x_2,\dots,\x_m)=\sum_{i=0}^kx_1^{(r)}x_1^{(r-1)}\dots x_1^{(r-i+1)}\tau_{k-i}^{(r-i)}(\x_2,\x_3,\dots,\x_m),\]
\[\bar{\sigma}_k^{(r)}(\x_1,\x_2,\dots,\x_m)=\sum_{i=0}^k\tau_{k-i}^{(r)}(\x_1,\x_2,\dots,\x_{m-1})x_m^{(r-k+i)}x_m^{(r-k+i-1)}\dots x_m^{(r-k+1)}.\]

\begin{example}
Let $n=4$.  Then
\begin{align*}\sigma_{5}^{(3)}(\textbf{x}_1,\textbf{x}_2)=&\ \tau_5^{(3)}(\x_2)+x_1^{(3)}\tau_4^{(2)}(\x_2)+x_1^{(3)}x_1^{(2)}\tau_3^{(1)}(\x_2)+x_1^{(3)}x_1^{(2)}x_1^{(1)}\tau_2^{(4)}(\x_2)+\\
&\ x_1^{(3)}x_1^{(2)}x_1^{(1)}x_1^{(4)}\tau_1^{(3)}(\x_2)+x_1^{(3)}x_1^{(2)}x_1^{(1)}x_1^{(4)}x_1^{(3)}\tau_0^{(2)}(\x_2)\\
=&\ x_1^{(3)}x_1^{(2)}\tau_3^{(1)}(\x_2)+x_1^{(3)}x_1^{(2)}x_1^{(1)}\tau_2^{(4)}(\x_2)+x_1^{(3)}x_1^{(2)}x_1^{(1)}x_1^{(4)}\tau_1^{(3)}(\x_2)+x_1^{(3)}x_1^{(2)}x_1^{(1)}x_1^{(4)}x_1^{(3)}\\
=&\ x_1^{(3)}x_1^{(2)}x_1^{(1)}x_1^{(4)}x_1^{(3)}+x_1^{(3)}x_1^{(2)}x_1^{(1)}x_1^{(4)}x_2^{(3)}+x_1^{(3)}x_1^{(2)}x_1^{(1)}x_2^{(4)}x_2^{(3)}+x_1^{(3)}x_1^{(2)}x_2^{(1)}x_2^{(4)}x_2^{(3)}.
\end{align*}

Similarly, $$\bar{\sigma}_{5}^{(3)}(\textbf{x}_1,\textbf{x}_2)=x_1^{(3)}x_1^{(2)}x_1^{(1)}x_2^{(4)}x_2^{(3)}+x_1^{(3)}x_1^{(2)}x_2^{(1)}x_2^{(4)}x_2^{(3)}+x_1^{(3)}x_2^{(2)}x_2^{(1)}x_2^{(4)}x_2^{(3)}+x_2^{(3)}x_2^{(2)}x_2^{(1)}x_2^{(4)}x_2^{(3)}.$$
\end{example}

We state a fundamental identity of the $\sigma$ and $\bar{\sigma}$ functions.
\begin{lemma} \label{lem:sigmaidentity}
\[\sigma_{(n-1)(j-i)}^{(r)}(\x_i, \dots, \x_j) = 
\sum_{k = 0}^{n-1}
\left ( \prod_{t = 0}^{k-1}x_i^{(r-t)} \right ) 
\sigma_{(n-1)(j-i-1)}^{(r-k)}(\x_i, \dots, \x_{j-1})
\left ( \prod_{s = 0}^{n-k-2}x_j^{(r-k+j-i-1-s)} \right ),
\]
\[
\bar{\sigma}_{(n-1)(j-i)}^{(r)}(\x_i, \dots, \x_j) = 
\sum_{k = 0}^{n-1}
\left ( \prod_{t = 0}^{k-1}x_i^{(r-t)} \right ) 
\bar{\sigma}_{(n-1)(j-i-1)}^{(r-k)}(\x_{i+1}, \dots, \x_{j})
\left ( \prod_{s = 0}^{n-k-2}x_j^{(r-k+j-i-1-s)} \right ).
\]
\end{lemma}
\begin{proof}
We sketch the proof of the first identity. The second identity is exactly dual. 

We can group the terms of $\sigma_{(n-1)(j-i)}^{(r)}(\x_i,\dots,\x_j)$ by the number of times $\x_j$ variables are used at the end. By definition of the $\sigma$ functions, $\x_j$ can appear at most $n-1$ times.
\begin{align*}
    \sigma_{(n-1)(j-i)}^{(r)}(\x_i, \dots, \x_j) 
    &= \sum_{k = 0}^{n-1}
\sigma_{(n-1)(j-i)-k}^{(r)}(\x_i, \dots, \x_{j-1})
\left ( \prod_{s = 0}^{k-1}x_j^{(r-k+j-i-1-s)} \right ) \\ 
    &= 
\sum_{k = 0}^{n-1}
\left ( \prod_{t = 0}^{k-1}x_i^{(r-t)} \right ) 
\sigma_{(n-1)(j-i-1)}^{(r-k)}(\x_i, \dots, \x_{j-1})
\left ( \prod_{s = 0}^{n-k-2}x_j^{(r-k+j-i-1-s)} \right ),
\end{align*}
Since all terms of $\sigma_{(n-1)(j-i)-k}^{(r-k)}(\x_i, \dots, \x_{j-1})$ must use $\x_i$ at least $n-1-k$ times, the second equality holds by a change of summation index from $k$ to $n-1-k$.
\end{proof}

\subsection{1-Shifts}\label{sec:oneshifts}
In this section, we state explicit formulas for the action of a permutation of the form $s_is_{i+1}\dots s_{j-1}$ and $s_{j-1}s_{j-2}\dots s_{i}$, where $1 \leq i < j \leq m$. Such permutations are shifts by 1 for $i<k<j$ so we call them 1-shifts.

\begin{theorem}[\cite{LP2010} Lemma 3.1] \label{thm:3.1}
Let $1 \leq i < j \leq m$. Then
\[
\kappa_r(s_{j-2}s_{j-3}\cdots s_{i}(\x_{j-1}),\x_j)
= 
\frac{
\sigma_{(n-1)(j-i)}^{(r-j+i)}(\x_i, \dots, \x_j)
}
{
\sigma_{(n-1)(j-i-1)}^{(r-j+i)}(\x_i, \dots,\x_{j-1})
}
\]
and
\[
s_{j-1}\dots s_i (x_j^{(r)}) 
=
\frac{
x_i^{(r-j+i)}\sigma_{(n-1)(j-i)}^{(r-j+i-1)}(\x_i, \dots, \x_j)
}{
\sigma_{(n-1)(j-i)}^{(r-j+i)}(\x_i, \dots, \x_j)
}.
\]
\end{theorem} 
The following lemma is the dual of \thmref{thm:3.1} and the proof exactly emulates the one in \cite{LP2010}.
\begin{theorem}[Dual of \thmref{thm:3.1}]\label{thm:3.1-dual}
Let $1 \leq i < j \leq m$. Then
\[
\kappa_r(
\x_i, 
s_{i+1} \dots s_{j-1}(\x_{i+1}))
= \frac{
\bar{\sigma}^{(r-1)}_{(n-1)(j-i)}(\x_i, \dots, \x_j)
}{
\bar{\sigma}^{(r)}_{(n-1)(j-i-1)}(\x_{i+1}, \dots, \x_j)
}
\]
and
\[s_i \dots s_{j-1}(x_i^{(r)}) 
=
\frac{
x_j^{(r+j-i)}\bar{\sigma}_{(n-1)(j-i)}^{(r)}(\x_i, \dots, \x_j)
}{
\bar{\sigma}_{(n-1)(j-i)}^{(r-1)}(\x_i, \dots, \x_j).
}\]
\end{theorem}
\begin{proof}
We prove the two statements in parallel by induction on $j - i$. For $j - i = 1$ they coincide with the formulae for the $\kappa_r$ and the $R$-action of $s_i$. By the induction assumption,
\[s_{i+1} \dots s_{j-1}(x_{i+1}^{(r)}) 
=
\frac{
x_j^{(r+j-i-1)}\bar{\sigma}_{(n-1)(j-i-1)}^{(r)}(\x_{i+1}, \dots, \x_j)
}{
\bar{\sigma}_{(n-1)(j-i-1)}^{(r-1)}(\x_{i+1}, \dots, \x_j)
}.\]

Therefore
\begin{align*}
    \kappa_r(
\x_i, 
s_{i+1}\dots s_{j-1}(\x_{i+1})) 
&= \sum_{s = 0}^{n-1}
\left [
\prod_{t = 1}^{s} s_{i+1}\dots s_{j-1}(x_{i+1}^{(r+t)})
\right ]
x_i^{(r+s+1)}\cdots x_i^{(r+n-1)} \\
&=\sum_{s = 0}^{n-1}
\left [
\prod_{t = 1}^{s} 
\frac{
x_j^{(r+t+j-i-1)}\bar{\sigma}_{(n-1)(j-i-1)}^{(r+t)}(\x_{i+1}, \dots, \x_j)
}{
\bar{\sigma}_{(n-1)(j-i-1)}^{(r+t-1)}(\x_{i+1}, \dots, \x_j)
}
\right ] x_i^{(r+s+1)}\cdots x_i^{(r+n-1)} \\ 
&= \sum_{s = 0}^{n-1} 
\left [ 
\prod_{t = 0}^{n-s-2} x_i^{(r+n-1-t)}
\right ]
\frac{
\bar{\sigma}_{(n-1)(j-i-1)}^{(r+s)}(\x_{i+1}, \dots, \x_j)
}{
\bar{\sigma}_{(n-1)(j-i-1)}^{(r)}(\x_{i+1}, \dots, \x_j)
}
\left [ 
\prod_{t = 0}^{s-1} x_{j}^{(r+s+j-i-1-t)}
\right ] \\
&= \frac{
\bar{\sigma}^{(r-1)}_{(n-1)(j-i)}(\x_i, \dots, \x_j)
}{
\bar{\sigma}^{(r)}_{(n-1)(j-i-1)}(\x_{i+1}, \dots, \x_j)
}.
\end{align*}
The last equality holds by \lemref{lem:sigmaidentity}. Now we can also prove the second claim, since
\begin{align*}
s_i \dots s_{j-1}(x_i^{(r)}) &=\frac{
    s_{i+1}\dots s_{j-1}(x_{i+1}^{(r+1)})
    \kappa_{r+1}(\x_i, s_{i+1}\dots s_{j-1}(\x_{i+1}))
}{
    \kappa_{r}(\x_i, s_{i+1}\dots s_{j-1}(\x_{i+1}))
} \\
&= x_j^{(r+j-i)} \frac{
\bar{\sigma}^{(r)}_{(n-1)(j-i)}(\x_{i}, \dots, \x_j)
}{
\bar{\sigma}^{(r-1)}_{(n-1)(j-i)}(\x_i, \dots, \x_j)
}.
\end{align*}
\end{proof}

We next consider what happens to $\x_i, \dots, \x_{j-1}$ under the action of $s_{j-1}\dots s_i$, and dually, what happens to $\x_{i+1}, \dots, \x_j$ under the action of $s_i \dots s_{j-1}$.
\begin{theorem} \label{lem:other_vars_cyc_shift}
Let $1 \leq i < j \leq m$. Then for $i \leq k < j$, 
\[
s_{j-1}\dots s_{i}(x_k^{(r)}) = 
\frac{x_{k+1}^{(r+1)}
\sigma_{(n-1)(k+1-i)}^{(r-k+i)}(\x_i, \dots, \x_{k+1})
\sigma_{(n-1)(k-i)}^{(r-k+i-1)}(\x_i, \dots, \x_k)
}{
\sigma_{(n-1)(k+1-i)}^{(r-k+i-1)}(\x_i, \dots, \x_{k+1})
\sigma_{(n-1)(k-i)}^{(r-k+i)}(\x_i, \dots, \x_k)
}.
\]
Similarly, for $i < k \leq j$,
\[
s_{i}\dots s_{j-1}(x_k^{(r)}) = 
\frac{x_{k-1}^{(r-1)}
\bar{\sigma}_{(n-1)(j-k+1)}^{(r-2)}(\x_{k-1}, \dots, \x_{j})
\bar{\sigma}_{(n-1)(j-k)}^{(r)}(\x_k, \dots, \x_j)
}{
\bar{\sigma}_{(n-1)(j-k+1)}^{(r-1)}(\x_{k-1}, \dots, \x_{j})
\bar{\sigma}_{(n-1)(j-k)}^{(r-1)}(\x_k, \dots, \x_j)
}.
\]
\end{theorem}

\begin{proof}
We prove the first part of the lemma. The second part is exactly dual. 

Let $s = s_{k-1}s_{k-2}\cdots s_{i}$. By \thmref{thm:3.1},
\[\kappa_r(s(\x_{k}),\x_{k+1})
= 
\frac{
\sigma_{(n-1)(k+1-i)}^{(r-k+i-1)}(\x_i, \x_{i+1}, \dots, \x_{k+1})
}
{
\sigma_{(n-1)(k-i)}^{(r-k+i-1)}(\x_i,\x_{i+1},\dots,\x_{k})
}.\]
So
\begin{align*}
    s_{j-1}s_{j-2}\dots s_i(x_k^{(r)}) &=
    s_{k}s_{k-1}\cdots s_i(x_{k}^{(r)}) \\
    &= s_k(s(x_k^{(r)})) \\
    &= x_{k+1}^{(r+1)}\frac{
    \kappa_{r+1}(s(\x_k), \x_{k+1})}{
    \kappa_{r}(s(\x_k), \x_{k+1})} \\ 
    &=x_{k+1}^{(r+1)}
    \frac{
    \sigma_{(n-1)(k+1-i)}^{(r-k+i)}(\x_i, \x_{i+1}, \dots, \x_{k+1})\sigma_{(n-1)(k-i)}^{(r-k+i-1)}(\x_i,\x_{i+1},\dots,\x_{k})
    }
    {
    \sigma_{(n-1)(k+1-i)}^{(r-k+i-1)}(\x_i, \x_{i+1}, \dots, \x_{k+1})\sigma_{(n-1)(k-i)}^{(r-k+i)}(\x_i,\x_{i+1},\dots,\x_{k})
    }.
\end{align*}
as desired.
\end{proof}

\subsection{Transpositions}\label{sec:trans}
In this section, we state formulas for the action of $s_k \dots s_{j-2}s_{j-1}s_{j-2}\dots s_i$, where $i \leq k < j$, and $s_{k-1} \dots s_{i+1}s_{i}s_{i+1}\dots s_{j-1}$, where $i < k \leq j$.  Note that this is the transposition that switches $i$ and $j$ when $k = i$ for the former permutation and $k = j$ for the latter. We assume throughout that $i < j-1$, as the case where $i = j-1$ is given by definition of the birational R-matrix action.

To state our formulas, we must first define the $\Omega$ functions.
\begin{definition}
For $i \leq k \leq j-1$, let
\[\Omega_{k}^{(r)}(\x_i, \dots, \x_j) = \sum_{\ell=0}^{n-1}\sigma^{(r)}_{(n-1)(k-i)+\ell}(\x_i,\dots,\x_k)\bar{\sigma}^{(r+k-i-\ell)}_{(n-1)(j-k)-\ell}(\x_{k+1},\dots,\x_j).\]
\end{definition}

\begin{example} The $\Omega$ functions generalize $\sigma$ and $\bar{\sigma}$ functions:
\begin{align*}
    \Omega_{j-1}^{(r)}(\x_i, \dots, \x_j) & = \sum_{\ell=0}^{n-1}\sigma^{(r)}_{(n-1)(j-i-1)+\ell}(\x_i,\dots,\x_{j-1})\bar{\sigma}^{(r+j-i-1-\ell)}_{n-1-\ell}(\x_j) \\ 
    & = \sum_{\ell=0}^{n-1}\sigma^{(r)}_{(n-1)(j-i-1)+\ell}(\x_i,\dots,\x_{j-1}) \prod_{t = 0}^{n-2-\ell}x_j^{(r-j-i-1-\ell-t)} \\ 
    & = \sigma^{(r)}_{(n-1)(j-i)}(\x_i, \dots, \x_j).
\end{align*}
Similarly, $\Omega_{i}^{(r)}(\x_i, \dots, \x_j) = \bar{\sigma}^{(r)}_{(n-1)(j-i)}(\x_i, \dots, \x_j)$.
\end{example}

Now we are ready to state our formulas for the action of $s_k \dots s_{j-2}s_{j-1}s_{j-2}\dots s_i$ and $s_{k-1} \dots s_{i+1}s_{i}s_{i+1}\dots s_{j-1}$.

\begin{theorem} \label{thm:trans_kappa} Let $s = s_k \dots s_{j-2}s_{j-1}s_{j-2}\dots s_i$. Then for $i < k < j$, 
\[
\kappa_r(s(x_{k-1}), s(x_{k})) = 
\frac{
\sigma_{(n-1)(k-i)}^{(r-k+i)}(\x_i, \dots, \x_k) \  \Omega_{k-1}^{(r-k+i)}(\x_i, \dots, \x_j)
}{
\sigma_{(n-1)(k-i-1)}^{(r-k+i)}(\x_i, \dots, \x_{k-1}) \  \Omega_{k}^{(r-k+i)}(\x_i, \dots, \x_j)
},
\]
and for $i \leq k < j$, 
\[
s(x_k^{(r)}) = x_j^{(r+j-k)}
\frac{
\sigma_{(n-1)(k-i)}^{(r-k+i-1)}(\x_i, \dots, \x_k) \ \Omega_{k}^{(r-k+i)}(\x_i, \dots,\x_j)
}{
\sigma_{(n-1)(k-i)}^{(r-k+i)}(\x_i, \dots, \x_k) \ \Omega_{k}^{(r-k+i-1)}(\x_i, \dots,\x_j)
}.
\]
\end{theorem}
\begin{theorem} \label{thm:trans_kappa_dual}
Let $s = s_{k-1}\dots s_{i+1}s_i s_{i+1} \dots s_{j-1}$. Then for $i < k < j$,
\[
\kappa_r(s(x_k), s(x_{k+1})) = 
\frac{\bar{\sigma}_{(n-1)(j-k)}^{(r-1)}(\x_k, \dots, \x_j) \Omega_k^{(r-k+i-1)}(\x_i, \dots, \x_j)}
{\bar{\sigma}_{(n-1)(j-k-1)}^{(r)}(\x_{k+1}, \dots, \x_j) \Omega_{k-1}^{(r-k+i-1)}(\x_i, \dots, \x_j)},
\]
and for $i < k \leq j$,
\[
s(x_k^{(r)}) = x_i^{(r-k+i)} \frac{\bar{\sigma}_{(n-1)(j-k)}^{(r)}(\x_k, \dots, \x_j)\Omega_{k-1}^{(r-k+i-1)}(\x_i, \dots, \x_j)}
{\bar{\sigma}_{(n-1)(j-k)}^{(r-1)}(\x_k, \dots, \x_j)\Omega_{k-1}^{(r-k+i)}(\x_i, \dots, \x_j)}.
\]
\end{theorem}

\begin{theorem} \label{thm:trans_formula} For $1 \leq i < j \leq m$ and $i < k < j$,
\[
s_i \dots s_{j-2} s_{j-1} s_{j-2} \dots s_i(x_k^{(r)}) = x_k^{(r)}
\frac{
\Omega_{k}^{(r-k+i)}(\x_i, \dots, \x_j) \ 
\Omega_{k-1}^{(r-k+i-1)}(\x_i, \dots, \x_j) 
}
{
\Omega_{k-1}^{(r-k+i)}(\x_i, \dots, \x_j) \ 
\Omega_{k}^{(r-k+i-1)}(\x_i, \dots, \x_j) 
}.
\]
\end{theorem}

Since Theorem \ref{thm:trans_kappa_dual} is entirely dual to Theorem \ref{thm:trans_kappa}, we will limit our discussion and proof in the remainder of the paper to the case of $s = s_k \dots s_{j-2}s_{j-1}s_{j-2}\dots s_i$. 

Note that \thmref{thm:trans_kappa} and \thmref{thm:trans_formula} solve the action of $s = s_k \dots s_{j-2}s_{j-1}s_{j-2}\dots s_i$ completely:
\begin{itemize}
    \item If $i=k$, then the action of $s$ on $\x_i$ is the same as the action of $s_{i}s_{i+1}\dots s_{j-1}$, which is given in the previous section.
    \item For $i \leq \ell < k$, the action of $s$ on $\x_\ell$ is the same as the action of $s_\ell s_{\ell-1} \dots s_i$, which is given in the previous section.
    \item If $i\neq k$, the action of $s$ on $\x_k$ is given by \thmref{thm:trans_kappa}.
    \item For $k < \ell < j$, the action of $s$ on $\x_\ell$ is the same as the action of $s_i \dots s_{j-2} s_{j-1} s_{j-2} \dots s_i$, which is solved by \thmref{thm:trans_formula}.
    \item The action of $s$ on $\x_j$ is the same as the action of $s_{j-1} \dots s_i$, which is given in the previous section.
\end{itemize}

The key ingredient to prove \thmref{thm:trans_kappa} and \thmref{thm:trans_formula} is the following identity.

\begin{lemma}
\label{thm:identity} For $i < k \leq j-1$, the following identity of $\Omega_{k-1}$ and $\Omega_k$ holds:
\begin{align*}
& \left ( \prod_{t = 1}^{n-1}\sigma_{(n-1)(k-i)}^{(r-k+i+t)}(\x_i, \dots, \x_k) \right )\ \Omega_{k-1}^{(r-k+i)}(\x_i, \dots, \x_j)\\
=&\sum_{s=0}^{n-1} 
\prod_{t=r+1}^{r+s}x_{j}^{(t+j-k)}\prod_{t=r+s+1}^{r+n-1}x_{k}^{(t+1)} 
\prod_{t = s+2}^{s+n-1} \sigma_{(n-1)(k-i)}^{(r-k+i+t)}(\x_i, \dots, \x_k)\ \\
& \Omega_{k}^{(r-k+i+s)}(\x_i, \dots, \x_j)
\sigma_{(n-1)(k-i-1)}^{(r-k+i+s+1)}(\x_i, \dots, \x_{k-1})
\end{align*}
\end{lemma} 

\begin{example}
Consider the case where $i = 1, j=r = n = 4$, and $k = 3$. Since $\Omega_3$ is the $\sigma$ function, this identity says that
\begin{align*}
    \ & \sigma_{6}^{(3)}(\x_1,\x_2,\x_3)\sigma_{6}^{(4)}(\x_1,\x_2,\x_3)\sigma_{6}^{(1)}(\x_1,\x_2,\x_3) 
    \Omega_{2}^{(2)}(\x_1,\x_2,\x_3,\x_4) \\ = \ & \x_3^{(2)}\x_3^{(3)}\x_3^{(4)}\sigma_{9}^{(2)}(\x_1,\x_2,\x_3,\x_4)\sigma_3^{(3)}(\x_1,\x_2)\sigma_{6}^{(4)}(\x_1,\x_2,\x_3)\sigma_{6}^{(1)}(\x_1,\x_2,\x_3) \\ 
    + \ & \x_4^{(2)}\x_3^{(3)}\x_3^{(4)}\sigma_{6}^{(2)}(\x_1,\x_2,\x_3)\sigma_{9}^{(3)}(\x_1,\x_2,\x_3,\x_4)\sigma_3^{(4)}(\x_1,\x_2)\sigma_{6}^{(1)}(\x_1,\x_2,\x_3) \\ 
    + \ & \x_4^{(2)}\x_4^{(3)}\x_3^{(4)}\sigma_{6}^{(2)}(\x_1,\x_2,\x_3)\sigma_{6}^{(3)}(\x_1,\x_2,\x_3)\sigma_{9}^{(4)}(\x_1,\x_2,\x_3,\x_4)\sigma_3^{(1)}(\x_1,\x_2) \\ 
    + \ & \x_4^{(2)}\x_4^{(3)}\x_4^{(4)}\sigma_3^{(2)}(\x_1,\x_2) \sigma_{6}^{(3)}(\x_1,\x_2,\x_3)\sigma_{6}^{(4)}(\x_1,\x_2,\x_3)\sigma_{9}^{(1)}(\x_1,\x_2,\x_3,\x_4). 
\end{align*}
\end{example}

The following proof assumes Lemma~\ref{thm:identity}, which is proven in Section~\ref{sec:transProof}.

\begin{proof}[Proof of \thmref{thm:trans_kappa}]
We proceed by induction on $k$. When $k = j-1$, by \thmref{lem:other_vars_cyc_shift}, indeed
\[
s(x_{j-1}^{(r)}) = s_{j-1}s_{j-2}\dots s_i(x_{j-1}^{(r)})
= x_j^{(r+1)}
\frac{
\sigma_{(n-1)(j-i-1)}^{(r-j+i)}(\x_i, \dots, \x_{j-1}) \ \sigma_{(n-1)(j-i)}^{(r-j+i+1)}(\x_i, \dots,\x_j)
}{
\sigma_{(n-1)(j-i-1)}^{(r-j+i+1)}(\x_i, \dots, \x_{j-1}) \ \sigma_{(n-1)(j-i)}^{(r-j+i)}(\x_i, \dots,\x_j)
}.
\]
Now let $s = s_k \dots s_{j-2}s_{j-1}s_{j-2}\dots s_i$ and suppose that 
\[
s(x_k^{(r)}) = x_j^{(r+j-k)}
\frac{
\sigma_{(n-1)(k-i)}^{(r-k+i-1)}(\x_i, \dots, \x_k) \ \Omega_{k}^{(r-k+i)}(\x_i, \dots,\x_j)
}{
\sigma_{(n-1)(k-i)}^{(r-k+i)}(\x_i, \dots, \x_k) \ \Omega_{k}^{(r-k+i-1)}(\x_i, \dots,\x_j)
}.
\]
By \thmref{lem:other_vars_cyc_shift}, 
\[s(x_{k-1}^{(r)}) = s_{j-1}s_{j-2}\dots s_i(x_{k-1}^{(r)}) = \frac{x_{k}^{(r+1)}
\sigma_{(n-1)(k-i)}^{(r-k+i+1)}(\x_i, \dots, \x_{k})
\sigma_{(n-1)(k-i-1)}^{(r-k+i)}(\x_i, \dots, \x_{k-1})
}{
\sigma_{(n-1)(k-i)}^{(r-k+i)}(\x_i, \dots, \x_{k})
\sigma_{(n-1)(k-i-1)}^{(r-k+i+1)}(\x_i, \dots, \x_{k-1})
}.\]
We may use the above to calculate the $\kappa$ function:
\begin{align*}
\kappa_r(s(\x_{k-1}),s(\x_{k}))
=& \sum_{s =0}^{n-1}\prod_{t=r+1}^{r+s}s(x_k^{(t)})\prod_{t = r+s+1}^{r+n-1}s(x_{k-1}^{(t)}) \\ 
=& \sum_{s = 0}^{n-1}\prod_{t=r+1}^{r+s}x_j^{(r+j-k)}
\frac{
\sigma_{(n-1)(k-i)}^{(t-k+i-1)}(\x_i, \dots, \x_k) \ \Omega_{k}^{(t-k+i)}(\x_i, \dots,\x_j)
}{
\sigma_{(n-1)(k-i)}^{(t-k+i)}(\x_i, \dots, \x_k) \ \Omega_{k}^{(t-k+i-1)}(\x_i, \dots,\x_j)
}\\ 
    & \prod_{t = r+s+1}^{r+n-1}\frac{x_{k}^{(t+1)}
\sigma_{(n-1)(k-i)}^{(t-k+i+1)}(\x_i, \dots, \x_{k})
\sigma_{(n-1)(k-i-1)}^{(t-k+i)}(\x_i, \dots, \x_{k-1})
}{
\sigma_{(n-1)(k-i)}^{(t-k+i)}(\x_i, \dots, \x_{k})
\sigma_{(n-1)(k-i-1)}^{(t-k+i+1)}(\x_i, \dots, \x_{k-1})
} \\ 
=& 
\sum_{s=0}^{n-1}
\left( \prod_{t = r+1}^{r+s}x_j^{(t+j-k)} \right) 
\left( \prod_{t = r+s+1}^{r+n-1}x_k^{(t+1)} \right) 
\frac{
\sigma_{(n-1)(k-i)}^{(r-k+i)}(\x_i, \dots, \x_k)
\ \Omega_{k}^{(r-k+i+s)}(\x_i, \dots,\x_j)
}{
\sigma_{(n-1)(k-i)}^{(r-k+i+s)}(\x_i, \dots, \x_k)
\ \Omega_{k}^{(r-k+i)}(\x_i, \dots,\x_j)
} \\
& \ \frac{
\sigma_{(n-1)(k-i)}^{(r-k+i)}(\x_i, \dots, \x_k)
\sigma_{(n-1)(k-i-1)}^{(r-k+i+s+1)}(\x_i, \dots,\x_{k-1})
}{
\sigma_{(n-1)(k-i)}^{(r-k+i+s+1)}(\x_i, \dots, \x_k)
\sigma_{(n-1)(k-i-1)}^{(r-k+i)}(\x_i, \dots,\x_{k-1})
} \\
= & \
\frac{
\sigma_{(n-1)(k-i)}^{(r-k+i)}(\x_i, \dots, \x_k)
\Omega_{k-1}^{(r-k+i)}(\x_i, \dots,\x_j)
}{
\sigma_{(n-1)(k-i-1)}^{(r-k+i)}(\x_i, \dots,\x_{k-1})\Omega_k^{(r-k+i)}(\x_i, \dots,\x_j)
}, 
\end{align*}
where the last equality is by Lemma~\ref{thm:identity}.
We can now calculate the action of $s$ on $\x_k$:
\begin{align*}
    & \ \ \ \ s_{k-1}s_k \dots s_{j-2}s_{j-1}s_{j-2}\dots s_i(x_{k-1}^{(r)}) \\ 
    &= s_{k-1}(s(x_{k-1}^{(r)})) \\ 
    &= s(x_k^{(r+1)})
    \frac{
    \kappa_{r+1}(s(\x_{k-1}), s(\x_{k}))
    }{
    \kappa_{r}(s(\x_{k-1}), s(\x_{k}))
    } \\ 
    &= x_j^{(r+j-k+1)}
\frac{
\sigma_{(n-1)(k-i)}^{(r-k+i)}(\x_i, \dots, \x_k) \ \Omega_{k}^{(r-k+i+1)}(\x_i, \dots,\x_j)
}{
\sigma_{(n-1)(k-i)}^{(r-k+i+1)}(\x_i, \dots, \x_k) \ \Omega_{k}^{(r-k+i)}(\x_i, \dots,\x_j)
} \\ 
& \ \frac{
\sigma_{(n-1)(k-i)}^{(r-k+i+1)}(\x_i, \dots, \x_k) \  \Omega_{k-1}^{(r-k+i+1)}(\x_i, \dots, \x_j)
}{
\sigma_{(n-1)(k-i-1)}^{(r-k+i+1)}(\x_i, \dots, \x_{k-1}) \  \Omega_{k}^{(r-k+i+1)}(\x_i, \dots, \x_j)
}
\frac{
\sigma_{(n-1)(k-i-1)}^{(r-k+i)}(\x_i, \dots, \x_{k-1}) \  \Omega_{k}^{(r-k+i)}(\x_i, \dots, \x_j)
}{
\sigma_{(n-1)(k-i)}^{(r-k+i)}(\x_i, \dots, \x_k) \  \Omega_{k-1}^{(r-k+i)}(\x_i, \dots, \x_j)
} \\
& = x_j^{(r+j-k+1)}
\frac{
\Omega_{k-1}^{(r-k+i+1)}(\x_i, \dots, \x_j)\sigma_{(n-1)(k-i-1)}^{(r-k+i)}(\x_i, \dots, \x_{k-1}) 
}{
\sigma_{(n-1)(k-i-1)}^{(r-k+i+1)}(\x_i, \dots, \x_{k-1})
\Omega_{k-1}^{(r-k+i)}(\x_i, \dots, \x_j)
}
\end{align*}
as desired.
\end{proof}
\begin{proof}[Proof of \thmref{thm:trans_formula}] Let $s = s_k \dots s_{j-2} s_{j-1} s_{j-2} \dots s_i$.
\begin{align*}
    s_i \dots s_{j-2} s_{j-1} s_{j-2} \dots s_i(x_k^{(r)}) 
    & = s_{k-1}s(x_k^{(r)}) \\ 
    & = s(x_{k-1}^{(r-1)})\frac{
    \kappa_{r-1}(s(\x_{k-1}), s(\x_{k}))
    }{
    \kappa_{r}(s(\x_{k-1}), s(\x_{k}))
    } \\ 
    & = s_{j-1} s_{j-2}\dots s_i(x_{k-1}^{(r-1)})\frac{
    \kappa_{r-1}(s(\x_{k-1}), s(\x_{k}))
    }{
    \kappa_{r}(s(\x_{k-1}), s(\x_{k}))
    }.
\end{align*}
Plugging in the formulas from \thmref{lem:other_vars_cyc_shift} and \thmref{thm:trans_kappa} yields the desired result.
\end{proof}

\section{Proof of Lemma~\ref{thm:identity}}
\label{sec:transProof}

We will need a series of technical lemmas to prove \lemref{thm:identity}.

To begin, we define a family of functions \[P_k^{(r)}(\x_i, \dots, \x_{j}) :=
\sum_{t = 0}^{k}
\prod_{s = 0}^{t-1} x_i^{(r-s)} \sigma_{(n-1)(j-i-1)}^{(r-t)}(\x_i, \dots, \x_{j-1})
\prod_{s = 0}^{k-t-1} x_{j}^{(r-t+j-i-1-s)}.
\] Notice that $P_k^{(r)}(\x_i, \dots, \x_{j})$ is the sum of all terms in $\sigma_{(n-1)(j-i-1)+k}^{(r)}(\x_i, \dots, \x_{j})$ that contain at most $k$ $\x_j$ variables.

We will also need the expressions $T_s$ for $0\leq s\leq n-1$, where \[T_s: = \prod_{t=r+s+1}^{r+n-1}x_{j}^{(t+1)} 
\prod_{t = s+2}^{s+n-1}
\sigma_{(n-1)(j-i)}^{(r-j+i+t)}(\x_i, \dots, \x_{j}) \sigma_{(n-1)(j-i-1)}^{(r-j+i+s+1)}(\x_i, \dots, \x_{j-1})
\prod_{t = 1}^{s}x_i^{(r-j+i+t)},\] 
and the sums
$S^k := \sum_{s=0}^{k}T_s$ and $S_k := \sum_{s = k}^{n-1}T_s$. The following lemma provides a product formula for these partial sums.
\begin{lemma} \label{lem:identity}
For $0 \leq k < n-1$,
\begin{align*}
S^k =
\prod_{t = r+k+1}^{r+n-1}x_{j}^{(t+1)}\prod_{t = 1}^{k}\sigma_{(n-1)(j-i)}^{(r-j+i+t)}(\x_i, \dots, \x_{j}) 
P_k^{(r-j+i+k+1)}(\x_i, \dots, \x_{j})
\prod_{t = k+2}^{n-1}\sigma_{(n-1)(j-i)}^{(r-j+i+t)}(\x_i, \dots, \x_{j}).
\end{align*}
For $0 < k \leq n-1$,
\begin{align*}
S_k = \prod_{t = 1}^{k}x_{i}^{(r-j+i+t)}
\prod_{t = k+1}^{n-1}\sigma_{(n-1)(j-i)}^{(r-j+i+t)}(\x_i, \dots, \x_{j}) 
P_{n-k-1}^{(r-j+i)}(\x_i, \dots, \x_{j})
\prod_{t = 1}^{k-1}\sigma_{(n-1)(j-i)}^{(r-j+i+t)}(\x_i, \dots, \x_{j}).
\end{align*}
And lastly,
\begin{align*}
    S^{n-1} = S_0 = \prod_{t = 1}^{n-1}\sigma_{(n-1)(j-i)}^{(r-j+i+t)}(\x_i, \dots, \x_{j}).
\end{align*}
\end{lemma}

\begin{example}
When $r = n = 4$, $j = 3$, $i = 1$,
\begin{align*}
    T_0 &= \x_3^{(2)}\x_3^{(3)}\x_3^{(4)}\sigma_3^{(3)}(\x_1,\x_2)\sigma_{6}^{(4)}(\x_1,\x_2,\x_3)\sigma_{6}^{(1)}(\x_1,\x_2,\x_3), \\ 
    T_1 &= \x_1^{(3)}\x_3^{(3)}\x_3^{(4)}\sigma_{6}^{(2)}(\x_1,\x_2,\x_3)\sigma_3^{(4)}(\x_1,\x_2)\sigma_{6}^{(1)}(\x_1,\x_2,\x_3), \\ 
    T_2 &= \x_1^{(4)}\x_1^{(3)}\x_3^{(4)}\sigma_{6}^{(2)}(\x_1,\x_2,\x_3)\sigma_{6}^{(3)}(\x_1,\x_2,\x_3)\sigma_3^{(1)}(\x_1,\x_2) \label{eq:idex_pt3}, \\ 
    T_3 &= \x_1^{(1)}\x_1^{(4)}\x_1^{(3)}\sigma_3^{(2)}(\x_1,\x_2) \sigma_{6}^{(3)}(\x_1,\x_2,\x_3)\sigma_{6}^{(4)}(\x_1,\x_2,\x_3).
\end{align*}
The lemma says that 
\begin{align*}
    T_0 &= \x_3^{(2)}\x_3^{(3)}\x_3^{(4)}P_0^{(3)}(\x_1,\x_2,\x_3)\sigma_{6}^{(4)}(\x_1,\x_2,\x_3)\sigma_{6}^{(1)}(\x_1,\x_2,\x_3), \\ 
    T_0 + T_1 &= \x_3^{(3)}\x_3^{(4)}\sigma_6^{(3)}(\x_1,\x_2,\x_3)P_1^{(4)}(\x_1,\x_2,\x_3)\sigma_6^{(1)}(\x_1,\x_2,\x_3), 
    \\ T_0+T_1+T_2 &= \x_3^{(4)}\sigma_6^{(3)}(\x_1,\x_2,\x_3)\sigma_6^{(4)}(\x_1,\x_2,\x_3)P_2^{(1)}(\x_1,\x_2,\x_3);
\end{align*}
\begin{align*}
    T_3 &= \x_1^{(1)}\x_1^{(4)}\x_1^{(3)}P_0^{(2)}(\x_1,\x_2,\x_3) \sigma_{6}^{(3)}(\x_1,\x_2,\x_3)\sigma_{6}^{(4)}(\x_1,\x_2,\x_3), \\ 
    T_2 + T_3 &= \x_1^{(4)}\x_1^{(3)}\sigma_6^{(1)}(\x_1,\x_2,\x_3)P_1^{(2)}(\x_1,\x_2,\x_3)\sigma_6^{(3)}(\x_1,\x_2,\x_3),
    \\ T_1+T_2+T_3 &= \x_1^{(3)}\sigma_6^{(4)}(\x_1,\x_2,\x_3)\sigma_6^{(1)}(\x_1,\x_2,\x_3)P_2^{(2)}(\x_1,\x_2,\x_3);
\end{align*}
and 
\[T_0+T_1+T_2+T_3 = \sigma_6^{(3)}(\x_1,\x_2,\x_3)\sigma_6^{(4)}(\x_1,\x_2,\x_3)\sigma_6^{(1)}(\x_1,\x_2,\x_3).\]
\end{example}

\begin{proof}
We prove the first claim first. We proceed by induction on $k$. When $k = 0$, the equality is by definition. Suppose that the claim is true for some $k$ such that $0 \leq k < n-2$. Then it suffices to prove that the proposed formula for $S^{k+1}$ satisfies
\[S^{k+1}= S^k + T_{k+1} .\] In other words, we need to show that
\begin{align*}
& \prod_{t = r+k+2}^{r+n-1}x_{j}^{(t+1)}\prod_{t = 1}^{k+1}\sigma_{(n-1)(j-i)}^{(r-j+i+t)}(\x_i, \dots, \x_{j}) 
P_{k+1}^{(r-j+i+k+2)}(\x_i, \dots, \x_{j})
\prod_{t = k+3}^{n-1}\sigma_{(n-1)(j-i)}^{(r-j+i+t)}(\x_i, \dots, \x_{j}) \\
= \ & \prod_{t = r+k+1}^{r+n-1}x_{j}^{(t+1)}\prod_{t = 1}^{k}\sigma_{(n-1)(j-i)}^{(r-j+i+t)}(\x_i, \dots, \x_{j}) 
P_k^{(r-j+i+k+1)}(\x_i, \dots, \x_{j})
\prod_{t = k+2}^{n-1}\sigma_{(n-1)(j-i)}^{(r-j+i+t)}(\x_i, \dots, \x_{j}) \\ 
+ &
\prod_{t=r+k+2}^{r+n-1}x_{j}^{(t+1)} 
\prod_{t = k+3}^{k+n} \sigma_{(n-1)(j-i)}^{(r-j+i+t)}(\x_i, \dots, \x_{j}) \sigma_{(n-1)(j-i-1)}^{(r-j+i+k+2)}(\x_i, \dots, \x_{j-1})
\prod_{t = 1}^{k+1}x_i^{(r-j+i+t)}. 
\end{align*}
We may factor out 
\[
\prod_{t = r+k+2}^{r+n-1}x_{j}^{(t+1)} \prod_{t=k+3}^{n-1}\sigma_{(n-1)(j-i)}^{(r-j+i+t)}(\x_i, \dots, \x_{j})
\prod_{t=1}^{k}\sigma_{(n-1)(j-i)}^{(r-j+i+t)}(\x_i, \dots, \x_{j})
\]
so that suffices to prove
\begin{align*}
& \sigma_{(n-1)(j-i)}^{(r-j+i+k+1)}(\x_i, \dots, \x_{j}) 
P_{k+1}^{(r-j+i+k+2)}(\x_i, \dots, \x_{j}) \\
= \ & x_{j}^{(r+k+2)}
P_k^{(r-j+i+k+1)}(\x_i, \dots, \x_{j}) \sigma_{(n-1)(j-i)}^{(r-j+i+k+2)}(\x_i, \dots, \x_{j}) \\ 
+ \ &
\sigma_{(n-1)(j-i)}^{(r-j+i)}(\x_i, \dots, \x_{j}) \sigma_{(n-1)(j-i-1)}^{(r-j+i+k+2)}(\x_i, \dots, \x_{j-1})
\prod_{t = 1}^{k+1}x_i^{(r-j+i+t)} .
\end{align*}

Note that we have the following identities:
\begin{align*}
\sigma_{(n-1)(j-i)}^{(r-j+i+k+1)}(\x_i, \dots, \x_{j}) 
& = P_{n-2}^{(r-j+i+k+1)}(\x_i, \dots, \x_j)x_j^{(r+k+2)} + 
\prod_{t = 0}^{n-2} x_i^{(r-j+i+k+1-t)} \sigma_{(n-1)(j-i-1)}^{(r-j+i+k+2)}(\x_i, \dots, \x_{j-1})
\\ 
& = P_k^{(r-j+i+k+1)}(\x_i, \dots, \x_j)\prod_{t = 0}^{n-2-k} x_{j}^{(r-t)} 
+ \prod_{t = 0}^{k} x_i^{(r-j+i+k+1-t)} P_{n-k-3}^{(r-j+i)} x_j^{(r+k+2)} \\
& \ + \prod_{t = 0}^{n-2} x_i^{(r-j+i+k+1-t)} \sigma_{(n-1)(j-i-1)}^{(r-j+i+k+2)}(\x_i, \dots, \x_{j-1})
\end{align*}
and 
\[P_{k+1}^{(r-j+i+k+2)} = x_i^{(r-j+i+k+2)}P_k^{(r-j+i+k+1)}(\x_i, \dots, \x_j) + \sigma_{(n-1)(j-i-1)}^{(r-j+i+k+2)}(\x_i, \dots, \x_{j-1}) 
\prod_{t = 0}^{k}x_{j}^{(r+k+1-t)}.\]

Using these two identities, we can rewrite the left hand side of the desired equality as the following sum:  \begin{align*}
& \sigma_{(n-1)(j-i)}^{(r-j+i+k+1)}(\x_i, \dots, \x_{j}) 
P_{k+1}^{(r-j+i+k+2)}(\x_i, \dots, \x_{j}) \\ 
    = \ & \prod_{t = 0}^{n-2} x_i^{(r-j+i+k+1-t)} \sigma_{(n-1)(j-i-1)}^{(r-j+i+k+2)}(\x_i, \dots, \x_{j-1})P_{k+1}^{(r-j+i+k+2)}(\x_i, \dots, \x_j) \\ + \ & 
\prod_{t = 0}^{k} x_i^{(r-j+i+k+1-t)} P_{n-k-3}^{(r-j+i)}(\x_i, \dots, \x_j) x_j^{(r+k+2)}
\sigma_{(n-1)(j-i-1)}^{(r-j+i+k+2)}(\x_i, \dots, \x_{j-1}) 
\prod_{t = 0}^{k}x_{j}^{(r+k+1-t)} \\
+ \ & P_k^{(r-j+i+k+1)}(\x_i, \dots, \x_j)\prod_{t = 0}^{n-2-k} x_{j}^{(r-t)} \sigma_{(n-1)(j-i-1)}^{(r-j+i+k+2)}(\x_i, \dots, \x_{j-1}) 
\prod_{t = 0}^{k}x_{j}^{(r+k+1-t)} \\ 
+ \ & P_{n-2}^{(r-j+i+k+1)}(\x_i, \dots, \x_j)x_j^{(r+k+2)}
x_i^{(r-j+i+k+2)}P_k^{(r-j+i+k+1)}(\x_i, \dots, \x_j). 
\end{align*}

One can check that the first two terms evaluate to \[\sigma_{(n-1)(j-i)}^{(r-j+i)}(\x_i, \dots, \x_{j}) \sigma_{(n-1)(j-i-1)}^{(r-j+i+k+2)}(\x_i, \dots, \x_{j-1})
\prod_{t = 1}^{k+1}x_i^{(r-j+i+t)}\] and the last two terms evaluate to \[x_j^{(r+k+2)}P_k^{(r-j+i+k+1)}(\x_i, \dots, \x_j)\sigma_{(n-1)(j-i)}^{(r-j+i+k+2)}(\x_i, \dots, \x_j),\] which shows the desired identity.

To prove the identity for $S_k$, we again proceed by induction on $k$. If $k = n-1$ the identity is true. Now suppose that the claim is true for some $k$ such that $1 < k \leq n-1$. We will show that the proposed formula for $S_{k-1}$ satisfies
\[S_{k-1} = T_{k-1} + S_k.\]

After cancelling common factors, the above is equivalent to
\begin{align*}
& P_{n-k}^{(r-j+i)}(\x_i, \dots, \x_j)\sigma_{(n-1)(j-i)}^{(r-j+i+k)}(\x_i, \dots, \x_j) \\ 
= \ &  \sigma_{(n-1)(j-i)}^{(r-j+i)}(\x_i, \dots, \x_j)\sigma_{(n-1)(j-i-1)}^{(r-j+i+k)}(\x_i, \dots, \x_{j-1})\prod_{t = k}^{n-1}x_{j}^{(r+t+1)}\\
+ \ & \sigma^{(r-j+i+k-1)}_{(n-1)(j-i)}(\x_i, \dots, \x_j)P_{n-k-1}^{(r-j+i)}(\x_i, \dots, \x_j)x_i^{(r-j+i+k)}.
\end{align*}
As before, we can expand the two factors on the left hand side.
\[
    P_{n-k}^{(r-j+i)}(\x_i, \dots, \x_j) 
    = 
    P_{n-k-1}^{(r-j+i)}(\x_i, \dots, \x_j)
    x_j^{(r+k)} 
    + 
    \prod_{t = 0}^{n-k-1}x_i^{(r-j+i-t)}
    \sigma_{(n-1)(j-i-1)}^{(r-j+i+k)}(\x_i, \dots, \x_{j-1})
\]
and
\begin{align*}
    \sigma_{(n-1)(j-i)}^{(r-j+i+k)}(\x_i, \dots, \x_j) = \ &
    x_i^{(r-j+i+k)}P_{n-2}^{(r-j+i+k-1)}(\x_i, \dots, \x_j) + \sigma_{(n-1)(j-i-1)}^{(r-j+i+k)}(\x_i, \dots, \x_{j-1})
    \prod_{t = 1}^{n-1}x_{j}^{(r+k-t)}
    \\
    = \ & \prod_{t = 0}^{k-1} x_i^{(r-j+i+k-t)}
    P_{n-k-1}^{(r-j+i)}(\x_i, \dots, \x_j)
    + 
    x_i^{(r-j+i+k)}
    P_{k-2}^{(r-j+i+k-1)}(\x_i, \dots, \x_j)
    \prod_{t = k}^{n-1}x_{j}^{(r+t+1)} \\
    + \ & \sigma_{(n-1)(j-i-1)}^{(r-j+i+k)}(\x_i, \dots, \x_{j-1})
    \prod_{t = 1}^{n-1}x_{j}^{(r+k-t)}.
\end{align*}
Thus, we can rewrite the left hand side into a sum of four terms, from which the desired equality follows.
\begin{align*}
    & P_{n-k}^{(r-j+i)}(\x_i, \dots, \x_j)\sigma_{(n-1)(j-i)}^{(r-j+i+k)}(\x_i, \dots, \x_j) \\ 
    = \ & P_{n-k}^{(r-j+i)}(\x_i, \dots, \x_j) \sigma_{(n-1)(j-i-1)}^{(r-j+i+k)}(\x_i, \dots, \x_{j-1})
    \prod_{t = 1}^{n-1}x_{j}^{(r+k-t)} \\
    + \ &
    \prod_{t = 0}^{n-k-1}x_i^{(r-j+i-t)}
    \sigma_{(n-1)(j-i-1)}^{(r-j+i+k)}(\x_i, \dots, \x_{j-1})
    x_i^{(r-j+i+k)}
    P_{k-2}^{(r-j+i+k-1)}(\x_i, \dots, \x_j)
    \prod_{t = k}^{n-1}x_{j}^{(r+t+1)} \\ 
    + \ & P_{n-k}^{(r-j+i)}(\x_i, \dots, \x_j) \prod_{t = 0}^{k-1} x_i^{(r-j+i+k-t)}
    P_{n-k-1}^{(r-j+i)}(\x_i, \dots, \x_j) \\ 
    + \ & P_{n-k-1}^{(r-j+i)}(\x_i, \dots, \x_j)
    x_j^{(r+k)} 
    x_i^{(r-j+i+k)}
    P_{k-2}^{(r-j+i+k-1)}(\x_i, \dots, \x_j)
    \prod_{t = k}^{n-1}x_{j}^{(r+t+1)} \\ 
    = \ & \sigma_{(n-1)(j-i)}^{(r-j+i)}(\x_i, \dots, \x_j)
    \sigma_{(n-1)(j-i-1)}^{(r-j+i+k)}(\x_i, \dots, \x_{j-1}) \prod_{t = k}^{n-1}x_{j}^{(r+t+1)} \\ 
    + \ & \sigma_{(n-1)(j-i)}^{(r-j+i+k-1)}(\x_i, \dots, \x_j) P_{n-k-1}^{(r-j+i)}(\x_i, \dots, \x_j) x_i^{(r-j+i+k)}.
\end{align*}

Lastly, we will show that \begin{align*}
    S^{n-1} = S_0 = \prod_{t = 1}^{n-1}\sigma_{(n-1)(j-i)}^{(r-j+i+t)}(\x_i, \dots, \x_{j}).
\end{align*}
It suffices to show that this product formula satisfies the identity \[S^{n-2} + T_{n-1} = S^{n-1},\] namely 
\begin{align*}
&
\prod_{t = r+n-1}^{r+n-1}x_{j}^{(t+1)}\prod_{t = 1}^{n-2}\sigma_{(n-1)(j-i)}^{(r-j+i+t)}(\x_i, \dots, \x_{j}) 
P_{n-2}^{(r-j+i+n-1)}(\x_i, \dots, \x_{j}) \\ 
+ \ & 
\prod_{t = 1}^{n-2}
\sigma_{(n-1)(j-i)}^{(r-j+i+t)}(\x_i, \dots, \x_{j}) \sigma_{(n-1)(j-i-1)}^{(r-j+i)}(\x_i, \dots, \x_{j-1})
\prod_{t = 1}^{n-1}x_i^{(r-j+i+t)} \\ 
= \ & \prod_{t = 1}^{n-1}\sigma_{(n-1)(j-i)}^{(r-j+i+t)}(\x_i, \dots, \x_{j}).
\end{align*}
This is true because
\begin{align*}
    x_{j}^{(r)}P_{n-2}^{(r-j+i+n-1)}(\x_i, \dots, \x_{j}) + \sigma_{(n-1)(j-i-1)}^{(r-j+i)}(\x_i, \dots, \x_{j-1})
\prod_{t = 1}^{n-1}x_i^{(r-j+i+t)} = \sigma_{(n-1)(j-i)}^{(r-j+i+n-1)}(\x_i, \dots, \x_{j}).
\end{align*}
\end{proof}
The following lemma evaluates two sums that we shall need in the proof of \lemref{thm:identity}. The proof relies heavily on \lemref{lem:identity}.
\begin{lemma} \label{lem:lhs=rhs} Let $0 \leq q \leq (n-1)(j-k+1)$ and $\gamma = 2(n-1)-(m-q)$. Then when $0 \leq m-q \leq n-1$,
\begin{align*}
    & \left [ \prod_{t = 1}^{n-1}\sigma_{(n-1)(k-i)}^{(r-k+i+t)}(\x_i, \dots, \x_k) \right ] P_{m-q}^{(r-k+i)}(\x_i, \dots, \x_k) \\ 
    = \ & 
    \sigma_{(n-1)(k-i)}^{(r-k+i-(m-q+1))}(\x_i, \dots, \x_k)
\sum_{s = (n-1)-(m-q)}^{n-1}
\prod_{t=r+s+1}^{r+n-1}x_{k}^{(t+1)} 
 \prod_{t = s+2}^{s+n-1} \sigma_{(n-1)(k-i)}^{(r-k+i+t)}(\x_i, \dots, \x_k) \\ &
    \prod_{\alpha = 0}^{s+m-q-n}x_i^{(r-k+i+s-\alpha)} 
 \sigma_{(n-1)(k-i-1)}^{(r-k+i+s+1)}(\x_i, \dots, \x_{k-1});
\end{align*}
and when $n-1\leq m-q \leq 2(n-1)$,
\begin{align*}
    & \left [ \prod_{t = 1}^{n-1}\sigma_{(n-1)(k-i)}^{(r-k+i+t)}(\x_i, \dots, \x_k) \prod_{\alpha =0}^{m-q-n}
x_i^{(r-k+i-\alpha)}
\prod_{\beta = n}^{m-q}
x_k^{(r-\beta)} \right ] P_{\gamma}^{(r-k+i-\gamma+1)}(\x_i, \dots, \x_k) \\ 
    = \ & 
    \sigma_{(n-1)(k-i)}^{(r-k+i-(m-q+1))}(\x_i, \dots, \x_k)
\sum_{s = 0}^{2(n-1)-(m-q)}
\prod_{t=r+s+1}^{r+n-1}x_{k}^{(t+1)} 
 \prod_{t = s+2}^{s+n-1} \sigma_{(n-1)(k-i)}^{(r-k+i+t)}(\x_i, \dots, \x_k) \\ &
    \prod_{\alpha = 0}^{s+m-q-n}x_i^{(r-k+i+s-\alpha)} 
 \sigma_{(n-1)(k-i-1)}^{(r-k+i+s+1)}(\x_i, \dots, \x_{k-1}),
\end{align*}
\end{lemma}
\begin{proof}
We begin with the first identity. Note that 
\begin{align*}
    \prod_{\alpha = 0}^{s+m-q-n}x_i^{(r-k+i+s-\alpha)} & = \prod_{\beta = (n-1)-(m-q)+1}^{s}x_i^{(r-k+i+\beta)} \\ & = \prod_{\beta = 1}^{s}x_i^{(r-k+i+\beta)} 
\left ( 
\prod_{\beta = 1}^{(n-1)-(m-q)}x_i^{(r-k+i+\beta)} 
\right )^{-1}.
\end{align*}
So we can rewrite the right hand side as:
\begin{align*}
    & \left ( 
\prod_{\beta = 1}^{(n-1)-(m-q)}x_i^{(r-k+i+\beta)} 
\right )^{-1} \sigma_{(n-1)(k-i)}^{(r-k+i-(m-q+1))}(\x_i, \dots, \x_k) \\ &
\sum_{s = (n-1)-(m-q)}^{n-1}
\prod_{t=r+s+1}^{r+n-1}x_{k}^{(t+1)} 
 \prod_{t = s+2}^{s+n-1} \sigma_{(n-1)(k-i)}^{(r-k+i+t)}(\x_i, \dots, \x_k)
 \sigma_{(n-1)(k-i-1)}^{(r-k+i+s+1)}(\x_i, \dots, \x_{k-1})
 \prod_{\beta = 1}^{s}x_i^{(r-k+i+\beta)}.
\end{align*}
Let $p = (n-1)-(m-q)$. Then by \lemref{lem:identity}, the sum itself is equal to
\[\prod_{\beta = 1}^{p}x_i^{(r-k+i+\beta)} \prod_{t = p+1}^{n-1}\sigma_{(n-1)(k-i)}^{(r-k+i+t)}(\x_i, \dots, \x_k)P_{n-1-p}^{(r-k+i)}(\x_i, \dots, \x_k)\prod_{t = 1}^{p-1}\sigma_{(n-1)(k-i)}^{(r-k+i+t)}(\x_i, \dots, \x_k).\] Substituting the sum, the right hand side is equal to 
\begin{align*}
    & \left ( 
\prod_{\beta = 1}^{p}x_i^{(r-k+i+\beta)} 
\right )^{-1} \sigma_{(n-1)(k-i)}^{(r-k+i+p)}(\x_i, \dots, \x_k) 
\prod_{\beta = 1}^{p}x_i^{(r-k+i+\beta)} \\ & \prod_{t = p+1}^{n-1}\sigma_{(n-1)(k-i)}^{(r-k+i+t)}(\x_i, \dots, \x_k)  P_{n-1-p}^{(r-k+i)}(\x_i, \dots, \x_k)\prod_{t = 1}^{p-1}\sigma_{(n-1)(k-i)}^{(r-k+i+t)}(\x_i, \dots, \x_k) \\ 
& = \prod_{t = 1}^{n-1}\sigma_{(n-1)(k-i)}^{(r-k+i+t)}(\x_i, \dots, \x_k) P_{m-q}^{(r-k+i)}(\x_i, \dots, \x_k)
\end{align*}
as desired.

For the second identity, similarly note that
\begin{align*}
    \prod_{\alpha = 0}^{s+m-q-n}x_i^{(r-k+i+s-\alpha)} & = \prod_{\alpha = 0}^{s-1}x_i^{(r-k+i+s-\alpha)} \prod_{\alpha = 0}^{m-q-n}x_i^{(r-k+i-\alpha)}\\ & = \prod_{\alpha = 1}^{s}x_i^{(r-k+i+\alpha)} \prod_{\alpha = 0}^{m-q-n}x_i^{(r-k+i-\alpha)}.
\end{align*}
So we can rewrite the right hand side as
\begin{align*}
    & \prod_{\alpha = 0}^{m-q-n}x_i^{(r-k+i-\alpha)} \sigma_{(n-1)(k-i)}^{(r-k+i-(m-q+1))}(\x_i, \dots, \x_k) \\ &
\sum_{s = 0}^{2(n-1)-(m-q)}
\prod_{t=r+s+1}^{r+n-1}x_{k}^{(t+1)} 
 \prod_{t = s+2}^{s+n-1} \sigma_{(n-1)(k-i)}^{(r-k+i+t)}(\x_i, \dots, \x_k)
 \sigma_{(n-1)(k-i-1)}^{(r-k+i+s+1)}(\x_i, \dots, \x_{k-1})
 \prod_{\alpha = 1}^{s}x_i^{(r-k+i+\alpha)}.
\end{align*}
By \lemref{lem:identity}, the sum is equal to 
\[
\prod_{t = r + \gamma + 1}^{r+n-1}x_k^{(t+1)}
\prod_{t = 1}^{\gamma}
\sigma_{(n-1)(k-i)}^{(r-k+i+t)}(\x_i, \dots, \x_k)
P_\gamma^{(r-k+i+\gamma+1)}(\x_i, \dots, \x_k)
\prod_{t = \gamma+2}^{n-1}
\sigma_{(n-1)(k-i)}^{(r-k+i+t)}(\x_i, \dots, \x_k),
\]
where we can rewrite the product of $\x_k$ variables as follows:
\[\prod_{t = r + \gamma + 1}^{r+n-1}x_k^{(t+1)}
= \prod_{\beta = 0}^{n-2-\gamma} x_k^{(r-\beta)} = \prod_{\beta = 0}^{m-q-n} x_k^{(r-\beta)} = \prod_{\beta = n}^{m-q} x_k^{(r-\beta)}.
\]
Thus, we get that the right hand side is equal to
\[
\prod_{t = 1}^{n-1}\sigma_{(n-1)(k-i)}^{(r-k+i+t)}(\x_i, \dots, \x_k) \prod_{\alpha = 0}^{m-q-n}x_i^{(r-k+i-\alpha)} 
\prod_{\beta = n}^{m-q} x_k^{(r-\beta)}
P_\gamma^{(r-k+i+\gamma+1)}(\x_i, \dots, \x_k)
\] as desired.
\end{proof}

We are now ready to prove Lemma~\ref{thm:identity}.
\begin{proof}[Proof of Lemma~\ref{thm:identity}.]
Notice that the only factor on the left hand side of this identity that contains $\x_j$ variables is the $\Omega$ function and the only factors on the right hand side that contain $\x_j$ variables is the $\Omega$ function and the product of $\x_j$ variables. We can subdivide terms of an $\Omega$ function according to the number of $\x_j$ variables contained in a term. In general,
\begin{align*}
    & \Omega^{(r)}_k(\x_i,\dots,\x_j) \\ = \ & 
    \sum_{\ell = 0}^{n-1}
    \sigma_{(n-1)(k-i)+\ell}^{(r)}(\x_i,\dots,\x_k)
    \bar{\sigma}_{(n-1)(j-k)-\ell}^{(r+k-i-\ell)}(\x_{k+1},\dots,\x_j) \\ 
    = \ & \sum_{\ell = 0}^{n-1}\sum_{q = n-1-\ell}^{(n-1)(j-k)-\ell}
    \prod_{\alpha = 0}^{\ell-1}x_i^{(r-\alpha)} 
    \sigma_{(n-1)(k-i)}^{(r-\ell)}(\x_i,\dots,\x_k)
    \tau^{(r+k-i-\ell)}_{(n-1)(j-k)-\ell-q}(\x_{k+1},\dots,\x_{j-1})\prod_{\beta = j+1}^{j+q}x_j^{(r+\beta-i)}.
\end{align*}
So we can rewrite the two $\Omega$ functions appearing in the identity as follows:
\begin{align*}
    \Omega_{k-1}^{(r-k+i)}(\x_i,\dots,\x_j) & = \sum_{\ell = 0}^{n-1}\sum_{q = n-1-\ell}^{(n-1)(j-k+1)-\ell}
    \prod_{\alpha = 0}^{\ell-1}x_i^{(r-k+i-\alpha)} 
    \sigma_{(n-1)(k-i-1)}^{(r-k+i-\ell)}(\x_i,\dots,\x_{k-1}) \\
    & \tau^{(r-\ell-1)}_{(n-1)(j-k+1)-\ell-q}(\x_{k},\dots,\x_{j-1})\prod_{\beta = j+1}^{j+q}x_j^{(r-k+\beta)} \\ 
    \Omega^{(r-k+i+s)}_{k}(\x_i,\dots,\x_j) 
    &= \sum_{\ell = 0}^{n-1}\sum_{q = n-1-\ell}^{(n-1)(j-k)-\ell}
    \prod_{\alpha = 0}^{\ell-1}x_i^{(r-k+i+s-\alpha)} 
    \sigma_{(n-1)(k-i)}^{(r-k+i+s-\ell)}(\x_i,\dots,\x_k) \\
    & \tau^{(r+s-\ell)}_{(n-1)(j-k)-\ell-q}(\x_{k+1},\dots,\x_{j-1})\prod_{\beta = j+1}^{j+q}x_j^{(r-k+s+\beta)}.
\end{align*}

If we substitute these expressions into the conjectured identity, then on both the left and right hand side the number of $\x_j$ variables ranges between $0$ and $(n-1)(j-k+1)$. We will now show that the terms with $q$ of the $\x_j$ variables on the left hand side are equal to the terms with $q$ of the $\x_j$ variables on the right hand side.

On the left hand side, if a term contains $q$ of the $\x_j$ variables, then these variables will be $\prod_{\beta = j+1}^{j+q}x_j^{(r-k+\beta)}$.  On the right hand side, if a term contains $q$ of the $\x_j$ variables, then these variables will be
\[\prod_{t = r+1}^{r+s}x_j^{(t+j-k)}\prod_{\beta = j+1}^{j+q-s}x_j^{(r-k+s+\beta)} = \prod_{\beta = j+1}^{j+q}x_j^{(r-k+\beta)}\] as well.

This means we can ignore this product of $\x_j$ variables in our calculations.

To finish our proof, we need to show that the following are equal for $0 \leq q \leq (n-1)(j-k+1)$:
\begin{equation}\label{eq:RHS_1red}
 \left [ \prod_{t = 1}^{n-1}\sigma_{(n-1)(k-i)}^{(r-k+i+t)}(\x_i, \dots, \x_k) \right ] \sum_{\ell = 0}^{n-1}
    \prod_{\alpha = 0}^{\ell-1}x_i^{(r-k+i-\alpha)} 
    \sigma_{(n-1)(k-i-1)}^{(r-k+i-\ell)}(\x_i,\dots,\x_{k-1})
    \tau^{(r-\ell-1)}_{(n-1)(j-k+1)-\ell-q}(\x_{k},\dots,\x_{j-1}),
\end{equation}
\begin{align}\label{eq:RHS_2red}
    \sum_{s=0}^{n-1} &
\prod_{t=r+s+1}^{r+n-1}x_{k}^{(t+1)} 
 \prod_{t = s+2}^{s+n-1} \sigma_{(n-1)(k-i)}^{(r-k+i+t)}(\x_i, \dots, \x_k) \\
 & \sum_{\ell = 0}^{n-1}
    \prod_{\alpha = 0}^{\ell-1}x_i^{(r-k+i+s-\alpha)} 
    \sigma_{(n-1)(k-i)}^{(r-k+i+s-\ell)}(\x_i,\dots,\x_k)
     \tau^{(r+s-\ell)}_{(n-1)(j-k)-\ell-q+s}(\x_{k+1},\dots,\x_{j-1}) \nonumber\\ 
& \sigma_{(n-1)(k-i-1)}^{(r-k+i+s+1)}(\x_i, \dots, \x_{k-1}). \nonumber
\end{align} 

We can simplify the problem further. Currently, the $\tau$ function in $\eqref{eq:RHS_1red}$ is a function of $\x_k, \dots \x_{j-1}$ whereas the $\tau$ function in $\eqref{eq:RHS_2red}$ is a function of $\x_{k+1}, \dots \x_{j-1}$. But recall that we can rewrite a $\tau$ function of $\x_k, \dots, \x_{j-1}$ in terms of $\tau$ functions of $\x_{k+1}, \dots \x_{j-1}$ like so:
\[\tau_{(n-1)(j-k+1)-\ell-q}^{(r-\ell-1)}(\x_k,\dots,\x_{j-1}) = \sum_{t = 0}^{n-1}
\prod_{\beta = 1}^{t}x_k^{(r-\ell-\beta)} 
\tau_{(n-1)(j-k+1)-\ell-q-t}^{(r-\ell-t-1)}(\x_{k+1 \dots j-1}).
\]
After making this substitution, the $\tau$ functions in the two equations will both be functions of $\x_{k+1}, \dots \x_{j-1}$ variables. Now note that the superscripts of all the $\tau$ functions match: If $\tau^{(*)}_{(n-1)(j-k+1)-m}(\x_{k+1},\dots,\x_{j-1})$ appears in either equation then the superscript is $r-(m-q+1)$.  This means that we can prove the identity by showing that the factors containing $\tau^{(r-(m-q+1))}_{(n-1)(j-k+1)-m}(\x_{k+1},\dots,\x_{j-1})$ in both equations are the same.

In the first equation, the terms that contain $\tau^{(r-(m-q+1))}_{(n-1)(j-k+1)-m}(\x_{k+1},\dots,\x_{j-1})$ appear in the sums indexed by $\ell$ and $t$ when $m=\ell+q+t$. Substituting $t = m-q-\ell$, we can simplify the nested sums into one by summing over $\ell$ such that $0 \leq \ell \leq n-1$ and $m-q-(n-1) \leq \ell \leq m-q$. Then the coefficient of $\tau^{(r-(m-q+1))}_{(n-1)(j-k+1)-m}(\x_{k+1},\dots,\x_{j-1})$ is
\begin{equation}\label{eq:LHS_3}
 \left [ \prod_{t = 1}^{n-1}\sigma_{(n-1)(k-i)}^{(r-k+i+t)}(\x_i, \dots, \x_k) \right ] \sum_{
 \substack{
 0 \leq \ell \leq n-1 \\ 
 m-q-(n-1) \leq \ell \leq m-q
 }
 }
    \prod_{\alpha = 0}^{\ell-1}x_i^{(r-k+i-\alpha)} 
    \sigma_{(n-1)(k-i-1)}^{(r-k+i-\ell)}(\x_i,\dots,\x_{k-1}) \prod_{\beta = 1}^{m-q-\ell}x_k^{(r-\ell-\beta)}.
\end{equation}
for $0\leq m-q\leq 2(n-1)$.

In the second equation, the terms that contain $\tau^{(r-(m-q+1))}_{(n-1)(j-k+1)-m}(\x_{k+1},\dots,\x_{j-1})$ appear in the sums indexed by $s$ and $\ell$ when $m-(n-1)=\ell+q-s$. Substituting $\ell = s + (m-q) - (n-1)$, we can simplify the nested sums into one by summing over $s$ such that $0 \leq s \leq n-1$, $(n-1)-(m-q) \leq s \leq 2(n-1)-(m-q)$.
Then the coefficient of $\tau^{(r-(m-q+1))}_{(n-1)(j-k+1)-m}(\x_{k+1},\dots,\x_{j-1})$ is

\begin{align}\label{eq:RHS_3}
    \sigma_{(n-1)(k-i)}^{(r-k+i-(m-q+1))}(\x_i, \dots, \x_k)
\sum_{
 \substack{
 0 \leq s \leq n-1 \\ 
 (n-1)-(m-q) \leq s \leq 2(n-1)-(m-q)
 }} &
\prod_{t=r+s+1}^{r+n-1}x_{k}^{(t+1)} 
 \prod_{t = s+2}^{s+n-1} \sigma_{(n-1)(k-i)}^{(r-k+i+t)}(\x_i, \dots, \x_k) \\ &
    \prod_{\alpha = 0}^{s+m-q-n}x_i^{(r-k+i+s-\alpha)} 
 \sigma_{(n-1)(k-i-1)}^{(r-k+i+s+1)}(\x_i, \dots, \x_{k-1}),\nonumber
\end{align}

To show that $\eqref{eq:LHS_3}$ and $\eqref{eq:RHS_3}$ are equal, we will consider two cases: 
\begin{itemize}
    \item $0 \leq m-q \leq n-1$;
    \item $n-1 \leq m-q \leq 2(n-1).$
\end{itemize} The two cases overlap when $m-q = n-1$, in which case the arguments for the two cases both apply.

In the first case, $\eqref{eq:LHS_3}$ is equal to 
\[
\left [ \prod_{t = 1}^{n-1}\sigma_{(n-1)(k-i)}^{(r-k+i+t)}(\x_i, \dots, \x_k) \right ] P_{m-q}^{(r-k+i)}(\x_i, \dots, \x_k).
\]
In the second case, $\eqref{eq:LHS_3}$ is equal to
\[
\left [ \prod_{t = 1}^{n-1}\sigma_{(n-1)(k-i)}^{(r-k+i+t)}(\x_i, \dots, \x_k) \prod_{\alpha =0}^{m-q-n}
x_i^{(r-k+i-\alpha)}
\prod_{\beta = n}^{m-q}
x_k^{(r-\beta)} \right ] P_{\gamma}^{(r-k+i-\gamma+1)}(\x_i, \dots, \x_k),\] where $\gamma = 2(n-1)-(m-q)$. \lemref{lem:lhs=rhs} shows equality in both cases.
\end{proof}

\section{Combinatorial interpretations}\label{sec:comb-interp}
Following Section 4.3 of \cite{LP2012}, let $N(n,m)$ be the cylindric grid network with $n$ horizontal wires and $m$ vertical loops where each crossing of a horizontal wire and a vertical loop is a vertex and all edges are oriented up and to the right (see Figure~\ref{fig:N(3,5)} for $N(3,5)$; to highlight the paths, the orientation of the edges in the network will be omitted in future figures).
The crossings of the $k$-th vertical loop are given weights of the form $x_k^{(r)}$ such that the upper indices of the vertex weights along around a vertical loop decrease by 1 at each crossing and the upper indices of the vertex weights along a horizontal wire increase by 1 at each crossing. As before, upper indices are taken mod $n$. The left and right endpoints of the horizontal wires are \textit{sources} and \textit{sinks} of $N(n,m)$, respectively. The sources and sinks inherit the upper indices of the closest crossing. 

\begin{figure}[H]
    \centering
    \begin{tikzpicture}[scale=1.3]
        \input{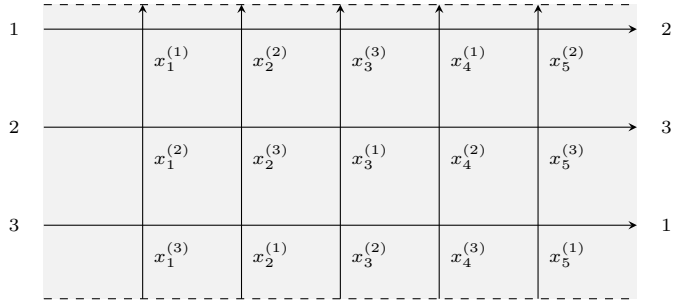}
    \end{tikzpicture}
    \caption{The network $N(3,5)$ with sources and sinks labelled. Note that the dashed top and bottom boundaries are identified.}
    \label{fig:N(3,5)}
\end{figure}

We write $p: s \to r$ to specify a path from source $s$ to sink $r$. A \textit{highway path} is a path from a source to a sink that never uses two up edges in a row (see Figure~\ref{fig:highway} for examples).  The \textit{weight} $\wt(p)$ of a highway path is the product of the weights of the vertices that it passes through when it has two right edges in a row. The \emph{degree} $\deg(p)$ is the degree of the monomial $\wt(p)$.

We can think of a highway path in $N(n,m)$ as a sequence of length $m$, consisting of \emph{through} steps, where the path crosses the vertical wire, and \emph{zigzags} where the path has one up step along the vertical wire. Call swapping an adjacent through step and zigzag in some path $p$ a \emph{switch}. We say a switch is \emph{allowed} if it does not cause the path to have multiple up steps in a row; if the switch is performed on a path in a path family, we require additionally that an allowed switch does not introduce any crossings.

We write $P: S \to R$ to specify a family of highway paths with source set $S$ and sink set $R$. In particular, we will be concerned with families of noncrossing highway paths (note that we allow paths in a noncrossing family to touch at corners). The \textit{weight} $\wt(P)$ of a family $P$ of noncrossing highway paths is the product of the weights of the paths in the family and the \emph{degree} $\deg(P)$ of such a family is the degree of the monomial $\wt(P)$. 

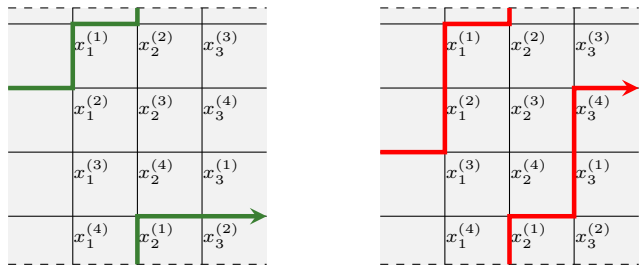
\begin{figure}[H]
    \centering
    \begin{minipage}{0.3\textwidth}
        \centering
        \begin{tikzpicture}[scale=0.85]
\fill[gray!10!white] (0,0.25) rectangle (4,4.25);
\path
(0,1) edge (4,1)
(0,2) edge (4,2)
(0,3) edge (4,3)
(0,4) edge (4,4)
(1,0.25) edge (1,4.25)
(2,0.25) edge (2,4.25)
(3,0.25) edge (3,4.25);
\path[dashed]
(0,0.25) edge (4,0.25)
(0,4.25) edge (4,4.25);
\node at (1.3,0.7) {\scriptsize$x_1^{(4)}$};
\node at (2.3,0.7) {\scriptsize$x_2^{(1)}$};
\node at (3.3,0.7) {\scriptsize$x_3^{(2)}$};
\node at (1.3,1.7) {\scriptsize$x_1^{(3)}$};
\node at (2.3,1.7) {\scriptsize$x_2^{(4)}$};
\node at (3.3,1.7) {\scriptsize$x_3^{(1)}$};
\node at (1.3,2.7) {\scriptsize$x_1^{(2)}$};
\node at (2.3,2.7) {\scriptsize$x_2^{(3)}$};
\node at (3.3,2.7) {\scriptsize$x_3^{(4)}$};
\node at (1.3,3.7) {\scriptsize$x_1^{(1)}$};
\node at (2.3,3.7) {\scriptsize$x_2^{(2)}$};
\node at (3.3,3.7) {\scriptsize$x_3^{(3)}$};
\draw[->, >=stealth,OliveGreen,line width=0.6mm]
(0,3) -- (1,3) -- (1,4) -- (2,4) -- (2,4.25)
(2,0.25) -- (2,1) -- (3,1) -- (4,1);
\end{tikzpicture}
    \end{minipage}
    \begin{minipage}{0.3\textwidth}
        \centering
       \begin{tikzpicture}[scale=0.85]
\fill[gray!10!white] (0,0.25) rectangle (4,4.25);
\path
(0,1) edge (4,1)
(0,2) edge (4,2)
(0,3) edge (4,3)
(0,4) edge (4,4)
(1,0.25) edge (1,4.25)
(2,0.25) edge (2,4.25)
(3,0.25) edge (3,4.25);
\path[dashed]
(0,0.25) edge (4,0.25)
(0,4.25) edge (4,4.25);
\node at (1.3,0.7) {\scriptsize$x_1^{(4)}$};
\node at (2.3,0.7) {\scriptsize$x_2^{(1)}$};
\node at (3.3,0.7) {\scriptsize$x_3^{(2)}$};
\node at (1.3,1.7) {\scriptsize$x_1^{(3)}$};
\node at (2.3,1.7) {\scriptsize$x_2^{(4)}$};
\node at (3.3,1.7) {\scriptsize$x_3^{(1)}$};
\node at (1.3,2.7) {\scriptsize$x_1^{(2)}$};
\node at (2.3,2.7) {\scriptsize$x_2^{(3)}$};
\node at (3.3,2.7) {\scriptsize$x_3^{(4)}$};
\node at (1.3,3.7) {\scriptsize$x_1^{(1)}$};
\node at (2.3,3.7) {\scriptsize$x_2^{(2)}$};
\node at (3.3,3.7) {\scriptsize$x_3^{(3)}$};
\draw[->, >=stealth,red,line width=0.6mm]
(0,2) -- (1,2) -- (1,3) -- (1,4) -- (2,4) -- (2,4.25)
(2,0.25) -- (2,1) -- (3,1) -- (3,2) -- (3,3) -- (4,3);
\end{tikzpicture}
    \end{minipage}
    \caption{A highway path with weight $x_3^{(2)}$ (left) and a non-highway path (right) in $N(4,3)$. The highway path on the left can be thought of as two zigzags steps and then a through step.}
    \label{fig:highway}
\end{figure}

\begin{remark}
Although previous sections worked with functions of $\x_i, \dots, \x_j$, in this section, we only give combinatorial interpretations $\tau, \sigma, \bar{\sigma}$ and $\Omega$ functions of $\x_1, \dots, \x_m$ to simplify notation. Correspondingly, the weights of crossings in $N(n,m)$ are variables in $\x_1, \dots, \x_m$. A combinatorial interpretation for functions of $\x_i, \dots, \x_j$ can be easily obtained by appropriately shifting the lower indices of the weights of the network $N(n,j-i+1)$. This is implicit in the proof of Theorem \ref{thm:comb_omega}, which uses the combinatorial interpretation of a $\bar{\sigma}$ function in the variables $\x_{k+1}, \dots, \x_m$.
\end{remark}

Let $\nye{N}(n,m)$ be the universal cover of $N(n,m)$ (see Figure \ref{fig:universalCover}).  Choose a lift of source 1 in $N(n,m)$ to label as source 1 in $\nye{N}(n,m)$.  Label the rest of the sources such that if a horizontal wire has source $s_i$ then the horizontal wire below has source $s_i+1$.  Label the sink of the horizontal wire with source $s_i$ as $s_i+m-1$. Note that every source and sink in $\nye{N}(n,m)$ has a label congruent modulo $n$ to the label of its projection in $N(n,m)$.

\begin{figure}
    \centering
    \begin{tikzpicture}
\fill[gray!10!white] (0,0.25) rectangle (7,6.25);
\path
(0,1) edge (7,1)
(0,2) edge (7,2)
(0,3) edge (7,3)
(0,4) edge (7,4)
(0,5) edge (7,5)
(0,6) edge (7,6)
(1,0.25) edge (1,6.25)
(2,0.25) edge (2,6.25)
(3,0.25) edge (3,6.25)
(4,0.25) edge (4,6.25)
(5,0.25) edge (5,6.25)
(6,0.25) edge (6,6.25);

\node at (1.3,0.7) {\scriptsize$x_1^{(2)}$};
\node at (2.3,0.7) {\scriptsize$x_2^{(1)}$};
\node at (3.3,0.7) {\scriptsize$x_3^{(2)}$};
\node at (4.3,0.7) {\scriptsize$x_4^{(1)}$};
\node at (5.3,0.7) {\scriptsize$x_5^{(2)}$};
\node at (6.3,0.7) {\scriptsize$x_6^{(1)}$};
\node at (1.3,1.7) {\scriptsize$x_1^{(1)}$};
\node at (2.3,1.7) {\scriptsize$x_2^{(2)}$};
\node at (3.3,1.7) {\scriptsize$x_3^{(1)}$};
\node at (4.3,1.7) {\scriptsize$x_4^{(2)}$};
\node at (5.3,1.7) {\scriptsize$x_5^{(1)}$};
\node at (6.3,1.7) {\scriptsize$x_6^{(2)}$};
\node at (1.3,2.7) {\scriptsize$x_1^{(2)}$};
\node at (2.3,2.7) {\scriptsize$x_2^{(1)}$};
\node at (3.3,2.7) {\scriptsize$x_3^{(2)}$};
\node at (4.3,2.7) {\scriptsize$x_4^{(1)}$};
\node at (5.3,2.7) {\scriptsize$x_5^{(2)}$};
\node at (6.3,2.7) {\scriptsize$x_6^{(1)}$};
\node at (1.3,3.7) {\scriptsize$x_1^{(1)}$};
\node at (2.3,3.7) {\scriptsize$x_2^{(2)}$};
\node at (3.3,3.7) {\scriptsize$x_3^{(1)}$};
\node at (4.3,3.7) {\scriptsize$x_4^{(2)}$};
\node at (5.3,3.7) {\scriptsize$x_5^{(1)}$};
\node at (6.3,3.7) {\scriptsize$x_6^{(2)}$};
\node at (1.3,4.7) {\scriptsize$x_1^{(2)}$};
\node at (2.3,4.7) {\scriptsize$x_2^{(1)}$};
\node at (3.3,4.7) {\scriptsize$x_3^{(2)}$};
\node at (4.3,4.7) {\scriptsize$x_4^{(1)}$};
\node at (5.3,4.7) {\scriptsize$x_5^{(2)}$};
\node at (6.3,4.7) {\scriptsize$x_6^{(1)}$};
\node at (1.3,5.7) {\scriptsize$x_1^{(1)}$};
\node at (2.3,5.7) {\scriptsize$x_2^{(2)}$};
\node at (3.3,5.7) {\scriptsize$x_3^{(1)}$};
\node at (4.3,5.7) {\scriptsize$x_4^{(2)}$};
\node at (5.3,5.7) {\scriptsize$x_5^{(1)}$};
\node at (6.3,5.7) {\scriptsize$x_6^{(2)}$};

\node at (-.5, 7) {\scriptsize $\vdots$};
\node at (-.5, 6) {\scriptsize $-1$};
\node at (-.5, 5) {\scriptsize $0$};
\node at (-.5, 4) {\scriptsize $1$};
\node at (-.5, 3) {\scriptsize $2$};
\node at (-.5, 2) {\scriptsize $3$};
\node at (-.5, 1) {\scriptsize $4$};
\node at (-.5, 0) {\scriptsize $\vdots$};

\node at (7.5, 7) {\scriptsize $\vdots$};
\node at (7.5, 6) {\scriptsize $4$};
\node at (7.5, 5) {\scriptsize $5$};
\node at (7.5, 4) {\scriptsize $6$};
\node at (7.5, 3) {\scriptsize $7$};
\node at (7.5, 2) {\scriptsize $8$};
\node at (7.5, 1) {\scriptsize $9$};
\node at (7.5, 0) {\scriptsize $\vdots$};

\node at (1, 7) {\scriptsize $\vdots$};
\node at (2, 7) {\scriptsize $\vdots$};
\node at (3, 7) {\scriptsize $\vdots$};
\node at (4, 7) {\scriptsize $\vdots$};
\node at (5, 7) {\scriptsize $\vdots$};
\node at (6, 7) {\scriptsize $\vdots$};

\node at (1, 0) {\scriptsize $\vdots$};
\node at (2, 0) {\scriptsize $\vdots$};
\node at (3, 0) {\scriptsize $\vdots$};
\node at (4, 0) {\scriptsize $\vdots$};
\node at (5, 0) {\scriptsize $\vdots$};
\node at (6, 0) {\scriptsize $\vdots$};
\end{tikzpicture}
    \caption{$\nye{N}(2,6)$.}
    \label{fig:universalCover}
\end{figure}

We now proceed to state and prove a combinatorial interpretation of $\tau$ functions. The following theorem is proven for certain cases in \cite{LP2012} through Lemma 6.5, which shows that $\tau^{(0)}_{(n-1)m-rn}(\x_1, \dots, \x_m)$ is a cylindric loop Schur function, and Proposition 4.7, which establishes a weight-preserving bijection between cylindric semistandard Young tableaux and families of noncrossing highway paths in $N(n,m)$ with specific source and sink sets. 

We extend these results to the generality of all $\tau$ functions by directly appealing to the properties of noncrossing paths.
\begin{theorem}\label{thm:comb_tau}
Let $k=\ell(n-1)+t$ where $0\leq t< n-1$.  Define $s_i=r-i+1$ and $$r_i=\begin{cases}s_i+\ell & i\leq t,\\s_i+\ell-1 & i>t,\end{cases}$$
where $i$ ranges from 1 to $n-1$. Let $\cP^{(r)}_{k}$ be the set of families $P = \{p_i: s_i\to r_i \}$ of noncrossing highway paths in $N(n,m)$ such that $\deg(P) = k$. Then
\[\tau_k^{(r)}(\x_1, \dots, \x_m)=\sum_{P \in \cP^{(r)}_{k}}\wt(P).\]
\end{theorem}
\begin{proof}
We begin by showing that each term of $\tau_k^{(r)}$ corresponds to a path family $P \in \cP^{(r)}_{k}$.  Consider the term in $\tau_k^{(r)}$ where each index in the sum is as low as possible: \[x_1^{(r)}x_1^{(r-1)}\dots x_1^{(r-n+2)}x_2^{(r-n+1)}\dots x_\ell^{(r-k+t+1)}x_{\ell+1}^{(r-k+t)}\dots x_{\ell+1}^{(r-k+1)}.\]
We get this \textit{initial term} as the weight of an \textit{initial family} of paths $P_0 = \{p_i: s_i \to r_i\}$, where $p_i$ is defined as follows. If $1 \leq i \leq t$, then the path $p_i$ goes through the first $\ell + 1$ crossings and zigzags until it reaches a sink. If $t < i \leq n-1$, then the path $p_i$ goes through the first $\ell$ crossings and similarly zigzags the rest of the way. This does not result in crossings (see Figure~\ref{fig:comb_tau}). We can compute which sink each path will end at by starting with the source it started at, increasing by 1 for each crossing it went straight through, and subtracting one at the end because the indices for the sinks are not shifted from the indices of the vertices in the previous column (zigzagging doesn't change the index). This shows that the sink of $p_i$ is indeed $r_i$.
\begin{figure}
        \centering
\begin{tikzpicture}
    \fill[gray!10!white] (-0.7,0.5) rectangle (7.7,9.5);
    \path[dashed]
    (-0.7,0.5) edge (7.7,0.5)
    (-0.7,9.5) edge (7.7,9.5);
    \path
    (0,1) edge (7,1)
    (0,2) edge (7,2)
    
    (0,3.5) edge (7,3.5)
    (0,4.5) edge (7,4.5)
    (0,5.5) edge (7,5.5)
    (0,6.5) edge (7,6.5)
    
    (0,8) edge (7,8)
    (0,9) edge (7,9)
    
    (1,0.5) edge (1,2.5)
    (1,3) edge (1,7)
    (1,7.5) edge (1,9.5)
    
    (3,0.5) edge (3,2.5)
    (3,3) edge (3,7)
    (3,7.5) edge (3,9.5)
    
    (5,0.5) edge (5,2.5)
    (5,3) edge (5,7)
    (5,7.5) edge (5,9.5);
    
    \node at (1.7,0.75) {\tiny$x_\ell^{(r+\ell-1)}$};
    \node at (1.7,1.75) {\tiny$x_\ell^{(r+\ell-2)}$};
    \node at (1.7,3.25) {\tiny$x_\ell^{(r+\ell-t)}$};
    \node at (1.9,4.25) {\tiny$x_\ell^{(r+\ell-t-1)}$};
    \node at (1.9,5.25) {\tiny$x_\ell^{(r+\ell-t-2)}$};
    \node at (1.9,6.25) {\tiny$x_\ell^{(r+\ell-t-3)}$};
    \node at (1.7,7.75) {\tiny$x_\ell^{(r+\ell+1)}$};
    \node at (1.55,8.75) {\tiny$x_{\ell}^{(r+\ell)}$};
    
    \node at (3.7,0.75) {\tiny$x_{\ell+1}^{(r+\ell)}$};
    \node at (3.7,1.75) {\tiny$x_{\ell+1}^{(r+\ell-1)}$};
    \node at (3.9,3.25) {\tiny$x_{\ell+1}^{(r+\ell-t+1)}$};
    \node at (3.7,4.25) {\tiny$x_{\ell+1}^{(r+\ell-t)}$};
    \node at (3.9,5.25) {\tiny$x_{\ell+1}^{(r+\ell-t-1)}$};
    \node at (3.9,6.25) {\tiny$x_{\ell+1}^{(r+\ell-t-2)}$};
    \node at (3.7,7.75) {\tiny$x_{\ell+1}^{(r+\ell+2)}$};
    \node at (3.7,8.75) {\tiny$x_{\ell+1}^{(r+\ell+1)}$};
    
    \node at (5.7,0.75) {\tiny$x_{\ell+2}^{(r+\ell+1)}$};
    \node at (5.7,1.75) {\tiny$x_{\ell+2}^{(r+\ell)}$};
    \node at (5.9,3.25) {\tiny$x_{\ell+2}^{(r+\ell-t+2)}$};
    \node at (5.9,4.25) {\tiny$x_{\ell+2}^{(r+\ell-t+1)}$};
    \node at (5.7,5.25) {\tiny$x_{\ell+2}^{(r+\ell-t)}$};
    \node at (5.9,6.25) {\tiny$x_{\ell+2}^{(r+\ell-t-1)}$};
    \node at (5.7,7.75) {\tiny$x_{\ell+2}^{(r+\ell+3)}$};
    \node at (5.7,8.75) {\tiny$x_{\ell+2}^{(r+\ell+2)}$};
    
    \node at (-0.3,1) {$\dots$};
    \node at (-0.3,2) {$\dots$};
    \node at (-0.3,3.5) {$\dots$};
    \node at (-0.3,4.5) {$\dots$};
    \node at (-0.3,5.5) {$\dots$};
    \node at (-0.3,6.5) {$\dots$};
    \node at (-0.3,8) {$\dots$};
    \node at (-0.3,9) {$\dots$};
    
    \node at (7.4,1) {$\dots$};
    \node at (7.4,2) {$\dots$};
    \node at (7.4,3.5) {$\dots$};
    \node at (7.4,4.5) {$\dots$};
    \node at (7.4,5.5) {$\dots$};
    \node at (7.4,6.5) {$\dots$};
    \node at (7.4,8) {$\dots$};
    \node at (7.4,9) {$\dots$};    

    \node at (1,2.85) {$\vdots$};
    \node at (3,2.85) {$\vdots$};
    \node at (5,2.85) {$\vdots$};
    \node at (4,8) {$\vdots$};
    
    \draw[->, >=stealth, line width=0.8mm,YellowGreen]
    (0,1) -- (5,1)--(5,2) -- (7,2);
    \draw[line width=0.8mm,OliveGreen]
    (0,2) -- (5,2) -- (5,2.5);
    
    \draw[->, >=stealth, line width=0.8mm,YellowGreen]
    (5,3) -- (5,3.5) -- (7,3.5);
    \draw[->, >=stealth, line width=0.8mm,OliveGreen]
    (0,3.5) -- (5,3.5) -- (5,4.5) -- (7,4.5);
    \draw[->, >=stealth, line width=0.8mm,YellowGreen]
    (0,4.5) -- (3,4.5) -- (3,5.5) -- (7,5.5);
    \draw[->, >=stealth, line width=0.8mm,OliveGreen]
    (0,5.5) -- (3,5.5) -- (3,6.5) -- (7,6.5);
    \draw[line width=0.8mm,YellowGreen]
    (0,6.5) -- (3,6.5)--(3,7);
    
    \draw[->, >=stealth, line width=0.8mm,OliveGreen]
    (3,7.5) -- (3,8) -- (7,8);
    \draw[->, >=stealth, line width=0.8mm,YellowGreen]
    (0,8) -- (3,8) -- (3,9) -- (7,9);
    
    \node at (-1.3,1) {\tiny$s_1 = r$};
    \node at (-1.5,2) {\tiny$s_2 = r-1$};
    \node at (-1.7,3.5) {\tiny$s_t = r-t+1$};
    \node at (-1.65,4.5) {\tiny$s_{t+1} = r-t$};
    \node at (-1.85,5.5) {\tiny$s_{t+2} = r-t-1$};
    \node at (-1.85,6.5) {\tiny$s_{t+3} = r-t-2$};
    \node at (-1.65,8) {\tiny$s_{n-1} = r+2$};
    \node at (-1.3,9) {\tiny$r+1$};
    
    \node at (7.4,1) {$\dots$};
    \node at (7.4,2) {$\dots$};
    \node at (7.4,3.5) {$\dots$};
    \node at (7.4,4.5) {$\dots$};
    \node at (7.4,5.5) {$\dots$};
    \node at (7.4,6.5) {$\dots$};
    \node at (7.4,8) {$\dots$};
    \node at (7.4,9) {$\dots$};   
    \end{tikzpicture}
    \caption{The initial family of paths $P_0$ in the proof of \thmref{thm:comb_tau}.}
    \label{fig:comb_tau}
\end{figure}
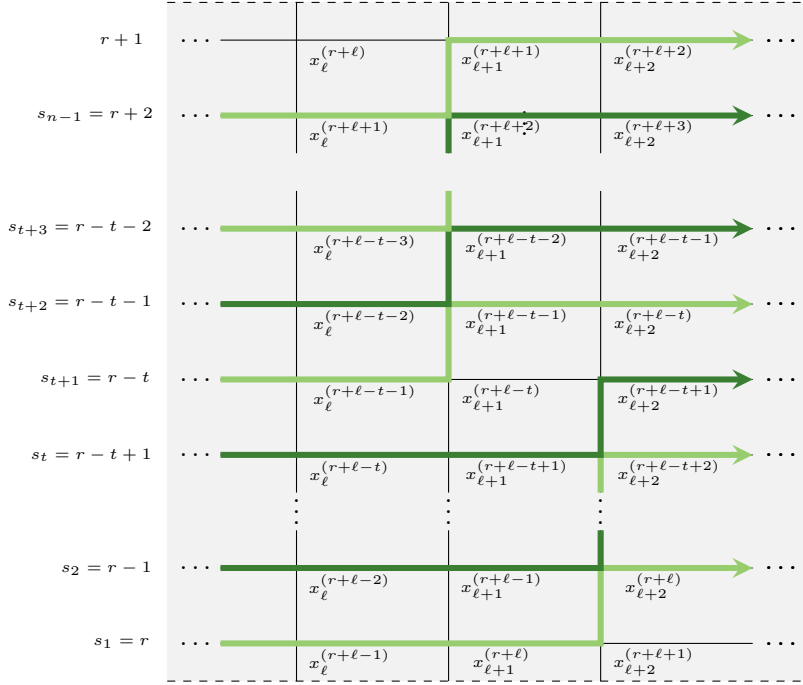

To show that each term is the weight of some family of noncrossing paths, we proceed by induction. 

We can get all other terms in $\tau_k^{(r)}$ by increasing the lower indices of this initial term one by one while maintaining the restrictions on the terms of $\tau$ at each step. Suppose that some term $x_{i_1}^{(r)}x_{i_2}^{(r-1)}\dots x_{i_k}^{(r-k+1)}$ in $\tau_k^{(r)}$ is the weight of a family $P$ of noncrossing highway paths with sources and sinks as described above, has $i_j=a$, and that shifting to $i_j=a+1$ gives another term in $\tau_k^{(r)}$. Note that if changing this index is allowed, this means $i_{j+1}>a$. It suffices to show that the new term is also the weight of a path family.

Since we can change the index of $i_j$, it must be that in $P$, the path that goes through the vertex with weight $x_a^{(r-j+1)}$ does not go through the next crossing. If it did, it would pick up the weight $x_{a+1}^{(r-j+2)}$. This would mean $\wt(P)$ must have been $\dots x_a^{(r-j+1)}x_{a+1}^{(r-j)}x_{a+1}^{(r-j-1)}\dots x_{a+1}^{(r-j-n+2)}\dots$, where $a+1$ appears as an index $n-1$ times. In this case, we would not be allowed to change $i_j$ from $a$ to $a+1$. So there is a path in $P$ that goes straight through $x_a^{(r-j+1)}$ and then zigzags at the next crossing.

We can apply a switch to the path through $x_a^{(r-j+1)}$ so that the path zigzags at $x_a^{(r-j+1)}$ and then goes straight through $x_{a+1}^{(r-j+1)}$. Since $i_{j+1}>a$ and $a$ appears as an index at most $n-1$ times, there is no path that goes through $x_a^{(r-j)}$. Hence this is an allowed switch. This gives us a family of highway paths that corresponds to the new term. 

Now we need to show that any path family $P \in \cP^{(r)}_k$ gives a term in $\tau_k^{(r)}$. An allowed switch on a path family $P$ replaces $x_i^{(a)}$ with $x_{i+1}^{(a)}$ in $\wt(P)$ when there is no $x_i^{(a-1)}$. So if $\wt(P)$ is a term in $\tau_k^{(r)}$, performing an allowed switch on $P$ generates a new term in $\tau_k^{(r)}$. Therefore, it suffices to show that any path family $P \in \cP^{(r)}_k$ is related to the initial family $P_0$ by a sequence of allowed switches. 

Consider the lifts $\nye{P}_o=\{p_i^*: s_i\to \nye{r_i}\}$ and $\nye{P}=\{p_i: s_i\to \nye{r_i}'\}$ in $\nye{N}(n,m)$ of $P$ and $P_0$. Since there are no crossings in $P$, all the sinks of $\nye{P}$ must have pairwise differences of less than $n$, and likewise for $\nye{P_0}$.  The sums of the sinks of $\nye{P}_0$ and $\nye{P}$ must be the same in order for the path families to have the same degrees, which means we must have $\nye{r_i}=\nye{r_i}'$, and $\deg(p_i) = \deg(p_i^*)$. 

We will now choose a sequence of allowed switches. So that we can better refer to the relative position of paths, consider the lift $\nye{P}$ of $P$ in $\nye{N}(n,m)$ such that the lifts of the sources are consecutive and the lowest (largest) source is $r$. Let loop $i$ be the first vertical loop where $\nye{P}$ differs from $\nye{P_0}$, and consider the lowest crossing $v$ where this difference occurs: a path $p$ zigzags at $v$ when the corresponding path $p^*$ in $P_0$ goes straight through $v$. Since $\deg(p) = \deg(p^*)$, there is at least one more through step in $p$. So we know that starting at $v$, our path $p$ zigzags at least once and then goes through some crossing $v'$. Between the crossing $v$ and $v'$, we will perform a sequence of switches starting at the crossing $v'$ and the previous crossing so that instead of consecutively zigzagging a number of times and then going through $v'$, this path will now first go through $v$ and then zigzag until it reaches $v'$. This will not introduce a crossing if $p$ is the path with source $r$, since there is no path in the family that has source $r+1$. It will also not introduce a crossing otherwise, because the path immediately below agrees with the corresponding path in the minimal family, which means that it goes through to the right at loop $i$.
\end{proof}

\begin{remark} \label{rem:altdefforPk}
The set of sources $S$ and set of sinks $R$ in Theorem~\ref{thm:comb_tau} have a clean presentation in terms of $r$ and $k$, namely $S = [n] \setminus \{r+1\}$ and $R = [n] \setminus \{r-k\}$. If two path families $P_1, P_2: S \to R$ lift to families with the same sources in the universal cover, they have the same degree if and only if their lifts have the same sink sets in the universal cover. The pairing between sources and sinks is then determined by the noncrossing property of the family. Thus we need not place restrictions on the pairing between sources and sinks and path families in $\cP_{k}^{(r)}$. In other words, $\cP_{k}^{(r)}$ is equivalently the set of noncrossing families of highway paths $P: S \to R$ such that $\deg(P) = k$.
\end{remark}
\begin{example} Consider $\tau_5^{(3)}(\x_1, \x_2)$ where $n = 4$. We have $k = 5, \ell = 1, t = 2$, and $r = 3$. So $s_1 = 3$, $s_2 = 2$, $s_3 = 1$, $r_1 = 4$, $r_2 = 3$, and $r_3 = 1$. There are two path families, $P_1, P_2$, that consist of highway paths $p_i: s_i \to r_i$, as in Figure~\ref{fig:tau-example}.
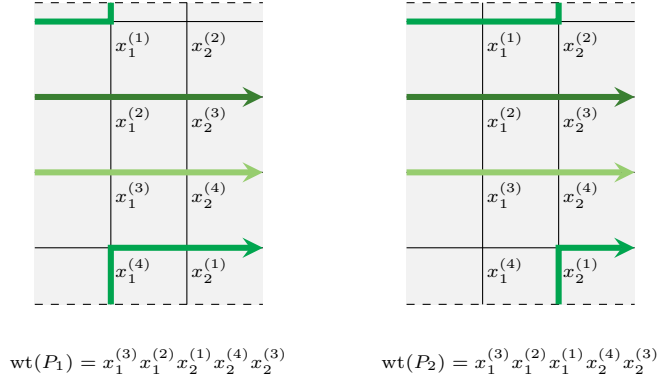
\begin{figure}
    \centering
    \begin{minipage}{0.3\textwidth}
    \begin{tikzpicture}
\fill[gray!10!white] (0,0.25) rectangle (3,4.25);
\path
(0,1) edge (3,1)
(0,2) edge (3,2)
(0,3) edge (3,3)
(0,4) edge (3,4)
(1,0.25) edge (1,4.25)
(2,0.25) edge (2,4.25);
\path[dashed]
(0,0.25) edge (3,0.25)
(0,4.25) edge (3,4.25);
\node at (1.3,0.7) {\scriptsize$x_1^{(4)}$};
\node at (2.3,0.7) {\scriptsize$x_2^{(1)}$};
\node at (1.3,1.7) {\scriptsize$x_1^{(3)}$};
\node at (2.3,1.7) {\scriptsize$x_2^{(4)}$};
\node at (1.3,2.7) {\scriptsize$x_1^{(2)}$};
\node at (2.3,2.7) {\scriptsize$x_2^{(3)}$};
\node at (1.3,3.7) {\scriptsize$x_1^{(1)}$};
\node at (2.3,3.7) {\scriptsize$x_2^{(2)}$};
\draw[->, >=stealth,YellowGreen,line width=0.8mm]
(0,2) -- (1,2) -- (2,2) -- (3,2);
\draw[->, >=stealth,OliveGreen,line width=0.8mm]
(0,3) -- (1,3) -- (2,3) -- (3,3);
\draw[->, >=stealth,Green,line width=0.8mm]
(0,4) -- (1,4) -- (1,4.25)
(1,0.25) -- (1,1) -- (2,1) -- (3,1);
\node at (1.5,-0.5) {\scriptsize$\wt(P_1) =x_1^{(3)}x_1^{(2)}x_2^{(1)} x_2^{(4)}x_2^{(3)}$};
\end{tikzpicture}
    \end{minipage}
    \begin{minipage}{0.3\textwidth}
    \begin{tikzpicture}[scale=1]
\fill[gray!10!white] (0,0.25) rectangle (3,4.25);
\path
(0,1) edge (3,1)
(0,2) edge (3,2)
(0,3) edge (3,3)
(0,4) edge (3,4)
(1,0.25) edge (1,4.25)
(2,0.25) edge (2,4.25);
\path[dashed]
(0,0.25) edge (3,0.25)
(0,4.25) edge (3,4.25);
\node at (1.3,0.7) {\scriptsize$x_1^{(4)}$};
\node at (2.3,0.7) {\scriptsize$x_2^{(1)}$};
\node at (1.3,1.7) {\scriptsize$x_1^{(3)}$};
\node at (2.3,1.7) {\scriptsize$x_2^{(4)}$};
\node at (1.3,2.7) {\scriptsize$x_1^{(2)}$};
\node at (2.3,2.7) {\scriptsize$x_2^{(3)}$};
\node at (1.3,3.7) {\scriptsize$x_1^{(1)}$};
\node at (2.3,3.7) {\scriptsize$x_2^{(2)}$};
\draw[->, >=stealth,YellowGreen,line width=0.8mm]
(0,2) -- (1,2) -- (2,2) -- (3,2);
\draw[->, >=stealth,OliveGreen,line width=0.8mm]
(0,3) -- (1,3) -- (2,3) -- (3,3);
\draw[->, >=stealth,Green,line width=0.8mm]
(0,4) -- (1,4) -- (2,4) -- (2,4.25)
(2,0.25) -- (2,1) -- (3,1);
\node at (1.5,-0.5) {\scriptsize$\wt(P_2) = x_1^{(3)}x_1^{(2)}x_1^{(1)}x_2^{(4)}x_2^{(3)}$};
\end{tikzpicture}
    \end{minipage}
    \caption{The two path families whose weights sum to $\tau^{(3)}_5(\x_1, \x_2)$}
    \label{fig:tau-example}
\end{figure}
In the example above, the path families consisting of $p_i: s_i \to r_i$ happen to have degree $5$. But the degree requirement becomes nontrivial in the following example. There is only one path family of degree $8$ with $p_1: 3 \to 3$, $p_2: 2 \to 2$, but there are more path families with $p_1: 3 \to 3$, $p_2: 2 \to 2$ but of lower degree (Figure \ref{fig:tau-why-degree}).
\begin{figure}
    \centering
    \begin{minipage}{0.4\textwidth}
    \begin{tikzpicture}
\fill[gray!10!white] (0,0.25) rectangle (5,3.25);
\path
(0,1) edge (5,1)
(0,2) edge (5,2)
(0,3) edge (5,3)
(1,0.25) edge (1,3.25)
(2,0.25) edge (2,3.25)
(3,0.25) edge (3,3.25)
(4,0.25) edge (4,3.25);
\path[dashed]
(0,0.25) edge (5,0.25)
(0,3.25) edge (5,3.25);
\node at (1.3,0.7) {\scriptsize$x_1^{(3)}$};
\node at (2.3,0.7) {\scriptsize$x_2^{(1)}$};
\node at (3.3,0.7) {\scriptsize$x_3^{(2)}$};
\node at (4.3,0.7) {\scriptsize$x_4^{(3)}$};
\node at (1.3,1.7) {\scriptsize$x_1^{(2)}$};
\node at (2.3,1.7) {\scriptsize$x_2^{(3)}$};
\node at (3.3,1.7) {\scriptsize$x_3^{(1)}$};
\node at (4.3,1.7) {\scriptsize$x_4^{(2)}$};
\node at (1.3,2.7) {\scriptsize$x_1^{(1)}$};
\node at (2.3,2.7) {\scriptsize$x_2^{(2)}$};
\node at (3.3,2.7) {\scriptsize$x_3^{(3)}$};
\node at (4.3,2.7) {\scriptsize$x_4^{(1)}$};
\draw[->, >=stealth,YellowGreen,line width=0.8mm]
(0,1) -- (1,1) -- (2,1) -- (3,1) -- (4,1) -- (5,1);
\draw[->, >=stealth,OliveGreen,line width=0.8mm]
(0,2) -- (1,2) -- (2,2) -- (3,2) -- (4,2) -- (5,2);
\node at (2.5,-0.5) {\scriptsize$x_1^{(3)}x_1^{(2)}x_2^{(1)}x_2^{(3)}x_3^{(2)}x_3^{(1)}x_4^{(3)}x_4^{(2)}$}; 
\end{tikzpicture}
    \end{minipage}
    \begin{minipage}{0.4\textwidth}
\begin{tikzpicture}
\fill[gray!5!white] (0,0.25) rectangle (5,3.25);
\path
(0,1) edge (5,1)
(0,2) edge (5,2)
(0,3) edge (5,3)
(1,0.25) edge (1,3.25)
(2,0.25) edge (2,3.25)
(3,0.25) edge (3,3.25)
(4,0.25) edge (4,3.25);
\path[dashed]
(0,0.25) edge (5,0.25)
(0,3.25) edge (5,3.25);
\node at (1.3,0.7) {\scriptsize$x_1^{(3)}$};
\node at (2.3,0.7) {\scriptsize$x_2^{(1)}$};
\node at (3.3,0.7) {\scriptsize$x_3^{(2)}$};
\node at (4.3,0.7) {\scriptsize$x_4^{(3)}$};
\node at (1.3,1.7) {\scriptsize$x_1^{(2)}$};
\node at (2.3,1.7) {\scriptsize$x_2^{(3)}$};
\node at (3.3,1.7) {\scriptsize$x_3^{(1)}$};
\node at (4.3,1.7) {\scriptsize$x_4^{(2)}$};
\node at (1.3,2.7) {\scriptsize$x_1^{(1)}$};
\node at (2.3,2.7) {\scriptsize$x_2^{(2)}$};
\node at (3.3,2.7) {\scriptsize$x_3^{(3)}$};
\node at (4.3,2.7) {\scriptsize$x_4^{(1)}$};
\draw[->, >=stealth,YellowGreen,line width=0.8mm]
(0,1) -- (1,1) -- (2,1) -- (2,2) -- (3,2) -- (3,3) -- (4,3) -- (4,3.25)
(4,0.25) -- (4,1) -- (5,1);
\draw[->, >=stealth,OliveGreen,line width=0.8mm]
(0,2) -- (1,2) -- (2,2) -- (2,3) -- (3,3) -- (3,3.25)
(3,0.25) -- (3,1) -- (4,1) -- (4,2) -- (5,2);
\node at (2.5,-0.5) {\scriptsize$x_1^{(3)}x_1^{(2)}$}; 
\end{tikzpicture}
 \end{minipage}
    \caption{Unique path family (left) that contributes to $\tau^{(3)}_8(\x_1, \x_2, \x_3, \x_4)$ and a path family (right) whose paths have the same sources and sinks}
    \label{fig:tau-why-degree}
\end{figure}
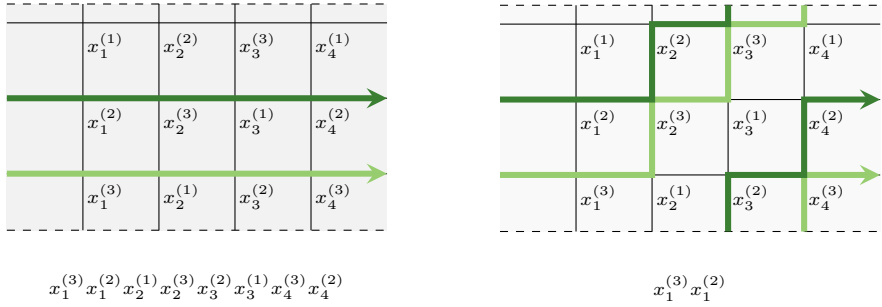
\end{example}

In order to state the definitions we need for the combinatorial interpretation of the $\sigma$ and $\bar{\sigma}$ functions, we must fist prove a lemma.

\begin{lemma}\label{lem:degree}
Fix a set of sources $S = \{s_1, s_2, \dots, s_k\}$ and a set of sinks $R = \{r_1, r_2, \dots, r_k\}$.  If $P,Q: S\to R$ are two families of highway paths in $N(n,m)$, then $\deg(P) \equiv_n \deg(Q)$.
\end{lemma}
\begin{proof}
We will first calculate $\deg(p)$ for a path $p: s\to r$ in $N(n,m)$ by considering its unique lift $\tilde{p}: s \to r-jn$.  Let $\tilde{p}_s$ be the path in $\nye{N}(n,m)$ that begins at source $s$, goes straight to the right, and ends at sink $s+m-1$.  We can see that $\deg(\tilde{p}_s)=m$. Given paths $\tilde{q}: s \to t-1$ and $\tilde{p}: s \to t$ in $\nye{N}(n,m)$, $\tilde{q}$ must have one more up step than $\tilde{p}$, and so $\deg(\tilde{q}) = \deg(\tilde{p})-1$.  This means we can calculate $$\deg(p)=m-(s+m-1-(r-jn))\equiv_n r-s+1.$$

Now consider a family of paths $P$ with source set $S = \{s_1, \dots, s_k\}$, $R = \{r_1, \dots, r_k\}$. By the previous paragraph, regardless of the pairing between sources and sinks, the sum of degrees of paths in $P$ is \[\deg(P) \equiv_n \sum_{i}r_i -  \sum_{j}s_j + k,\] which concludes our proof.
\end{proof}

Let $k \leq m(n-1)$, $S = [n]\setminus\{r+1\}$, and $R = [n]\setminus\{r-k\}$ and let $\cP^{(r)}_{\leq k}$ be the set of families $P: S \to R$ of noncrossing highway paths in $N(n,m)$ such that $\deg(P) \leq k$. In the proof of Theorem~\ref{thm:comb_tau}, for $k \leq m(n-1)$, we exhibited a family of noncrossing paths from $S$ to $R$ of degree $k$. So by Lemma \ref{lem:degree}, any family $P: S\to R$ of noncrossing highway in $N(n,m)$ has $\deg(P)=k-jn$ for some integer $j$. Thus we can define 
\[\wt_{\sigma_k}(P) = \left(\prod_{i=0}^{n-1}x_1^{(i)}\right)^j \wt(P) \text{ and } 
\wt_{\bar{\sigma}_k}(P) = \left(\prod_{i=0}^{n-1}x_m^{(i)}\right)^j \wt(P).\] 
\begin{theorem} \label{thm:comb_sigma}
If $k\leq m(n-1)$, then
\[\sigma_k^{(r)}(\x_1, \dots, \x_m)=
\sum_{P \in \cP^{(r)}_{\leq k}}
\wt_{\sigma_k}(P), \text{ \ \  }
\bar{\sigma}_k^{(r)}(\x_1, \dots, \x_m)=
\sum_{P \in \cP^{(r)}_{\leq k}}
\wt_{\bar{\sigma}_k}(P).
\]
\end{theorem}

\begin{proof}
Let $k = an+b$ where $0 \leq b < n$.

We will show that $\cP_{\leq k}^{(r)} = \bigcup_{j = 0}^{a} \cP_{k-jn}^{(r)}$.
By checking that the set of sources and sinks that define $\cP_{k-jn}^{(r)}$ are exactly $S$ and $R$, we note that $\bigcup_{j = 0}^{a} \cP_{k-jn}^{(r)} \subseteq \cP_{\leq k}^{(r)}$. Since the possible degrees for path families in $\cP_{\leq k}^{(r)}$ are precisely $k-jn$ where $0 \leq j \leq a$, by Remark \ref{rem:altdefforPk}, we have $\cP_{\leq k}^{(r)} \subseteq \bigcup_{j = 0}^{a} \cP_{k-jn}^{(r)}$. Therefore,
\begin{align*}
    \sigma_k^{(r)}(\x_1, \dots, \x_m)  &=\sum_{j=0}^a\left(\prod_{i=0}^{n-1}x_1^{(i)}\right)^j
\sum_{P \in \cP^{(r)}_{k-jn}} \wt(P) \\ 
&=\sum_{j=0}^a\sum_{P \in \cP^{(r)}_{k-jn}} \wt_{\sigma_k}(P) \\
&= \sum_{P \in \cP^{(r)}_{\leq k}} \wt_{\sigma_k}(P).
\end{align*}

The proof of the second part of the theorem concerning $\bar{\sigma}$ is entirely analogous.
\end{proof}

\begin{remark}
When $k > m(n-1)$, we can write 
\[\sigma_{k}^{(r)}(\x_1, \dots, \x_m) = \prod_{t = 0}^{k-m(n-1)-1}x_1^{(r-t)}\sigma_{m(n-1)}^{(r-k+m)}(\x_1, \dots, \x_m)\]
and 
\[\bar{\sigma}_{k}^{(r)}(\x_1, \dots, \x_m) = \bar{\sigma}_{m(n-1)}^{(r)}(\x_1, \dots, \x_m) \prod_{t = 0}^{k-m(n-1)-1}x_m^{(r+m-t)}.\] Thus we can adjust the weight of every family of paths that arises as a term of $\sigma_{m(n-1)}^{(r-k+m)}(\x_1, \dots, \x_m)$ or $\bar{\sigma}_{m(n-1)}^{(r-k+m)}(\x_1, \dots, \x_m)$ monomial to obtain a combinatorial interpretation for the case where $k > m(n-1)$.
\end{remark}

\begin{example} We apply Theorem \ref{thm:comb_sigma} to $\sigma_5^{(3)}(\x_1, \x_2)$ with $n = 4$. We have $k = 5, a = 1, b = 1, r = 3$. So $S = \{1,2,3\}$ and $R = \{1,3,4\}$. There are four possible path families $P_1, P_2, P_3, P_4$ with source set $S$ and sink set $R$ (see Figure \ref{fig:tau-example} and Figure \ref{fig:sigma-example}). We can compute $\wt_{\sigma}(P_i)$ for each path family: for $i = 1,2$, since $\deg(P_i) = 5$, $\wt_{\sigma_5}(P_i) = \wt(P_i)$; for $i = 3,4$, since $\deg(P_i) = 1$, $\wt_{\sigma_5}(P_i) = x_1^{(3)}x_1^{(2)}x_1^{(1)}x_1^{(4)}\wt(P_i)$. Indeed, 
\[\sigma_5^{(3)}(\x_1, \x_2) = x_1^{(3)}x_1^{(2)}x_2^{(1)}x_2^{(4)}x_2^{(3)}+
x_1^{(3)}x_1^{(2)}x_1^{(1)}x_2^{(4)}x_2^{(3)}+
x_1^{(3)}x_1^{(2)}x_1^{(1)}x_1^{(4)}x_1^{(3)}+
x_1^{(3)}x_1^{(2)}x_1^{(1)}x_1^{(4)}x_2^{(3)}.
\]
\begin{figure}[H]
    \centering
    \begin{minipage}{0.3\textwidth}
    \begin{tikzpicture}
\fill[gray!10!white] (0,0.25) rectangle (3,4.25);
\path
(0,1) edge (3,1)
(0,2) edge (3,2)
(0,3) edge (3,3)
(0,4) edge (3,4)
(1,0.25) edge (1,4.25)
(2,0.25) edge (2,4.25);
\path[dashed]
(0,0.25) edge (3,0.25)
(0,4.25) edge (3,4.25);
\node at (1.3,0.7) {\scriptsize$x_1^{(4)}$};
\node at (2.3,0.7) {\scriptsize$x_2^{(1)}$};
\node at (1.3,1.7) {\scriptsize$x_1^{(3)}$};
\node at (2.3,1.7) {\scriptsize$x_2^{(4)}$};
\node at (1.3,2.7) {\scriptsize$x_1^{(2)}$};
\node at (2.3,2.7) {\scriptsize$x_2^{(3)}$};
\node at (1.3,3.7) {\scriptsize$x_1^{(1)}$};
\node at (2.3,3.7) {\scriptsize$x_2^{(2)}$};
\draw[->, >=stealth,YellowGreen,line width=0.8mm]
(0,2) -- (1,2) -- (2,2) -- (2,3) -- (3,3);
\draw[->, >=stealth,OliveGreen,line width=0.8mm]
(0,3) -- (1,3) -- (1,4) -- (2,4) -- (2,4.25)
(2,0.25) -- (2,1) -- (3,1);
\draw[->, >=stealth,Green,line width=0.8mm]
(0,4) -- (1,4) -- (1,4.25)
(1,0.25) -- (1,1) -- (2,1) -- (2,2) -- (3,2);
\node at (1.5,-0.5) {\scriptsize$\wt(P_3) =x_1^{(3)}$}; 
\end{tikzpicture}
    \end{minipage}
    \begin{minipage}{0.3\textwidth}
\begin{tikzpicture}
\fill[gray!10!white] (0,0.25) rectangle (3,4.25);
\path
(0,1) edge (3,1)
(0,2) edge (3,2)
(0,3) edge (3,3)
(0,4) edge (3,4)
(1,0.25) edge (1,4.25)
(2,0.25) edge (2,4.25);
\path[dashed]
(0,0.25) edge (3,0.25)
(0,4.25) edge (3,4.25);
\node at (1.3,0.7) {\scriptsize$x_1^{(4)}$};
\node at (2.3,0.7) {\scriptsize$x_2^{(1)}$};
\node at (1.3,1.7) {\scriptsize$x_1^{(3)}$};
\node at (2.3,1.7) {\scriptsize$x_2^{(4)}$};
\node at (1.3,2.7) {\scriptsize$x_1^{(2)}$};
\node at (2.3,2.7) {\scriptsize$x_2^{(3)}$};
\node at (1.3,3.7) {\scriptsize$x_1^{(1)}$};
\node at (2.3,3.7) {\scriptsize$x_2^{(2)}$};
\draw[->, >=stealth,YellowGreen,line width=0.8mm]
(0,2) -- (1,2) -- (1,3) -- (2,3) -- (3,3);
\draw[->, >=stealth,OliveGreen,line width=0.8mm]
(0,3) -- (1,3) -- (1,4) -- (2,4) -- (2,4.25)
(2,0.25) -- (2,1) -- (3,1);
\draw[->, >=stealth,Green,line width=0.8mm]
(0,4) -- (1,4) -- (1,4.25)
(1,0.25) -- (1,1) -- (2,1) -- (2,2) -- (3,2);
\node at (1.5,-0.5) {\scriptsize$\wt(P_4) = x_2^{(3)}$};   
\end{tikzpicture}
 \end{minipage}
    \caption{The two path families whose weights sum to $\tau^{(3)}_1(\x_1, \x_2)$}
    \label{fig:sigma-example}
\end{figure}
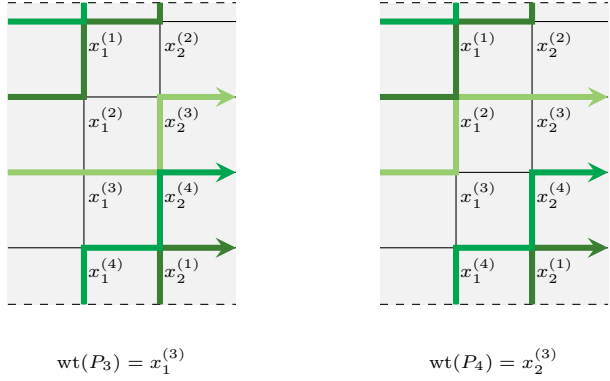

\end{example}
Based on the interpretations of $\sigma$ and $\bar{\sigma}$ functions, we obtain a combinatorial interpretation of the $\Omega$ functions.

Let $S = [n]\setminus\{r+1\}$ and $R = [n] \setminus \{r+m-1\}$ be the sources and sinks of $P \in \cP^{(r)}_{(m-1)(n-1)}$ in $N(n,m)$. Given a family of noncrossing highway paths from $S$ to $R$, cut the network along the middle of the $\x_k$ and $\x_{k+1}$ vertical loops, resulting in two families $P_1$ and $P_2$ of noncrossing highway paths in $N(n,k)$ and $N(n,m-k)$ respectively.  Then there exists $\ell$ such that $R' = [n] \setminus \{r+k-1-\ell\}$ is the sink set of $P_1$ and $S' = [n] \setminus \{r+k-\ell\}$ is the source set of $P_2$. Since $(k-1)(n-1)+\ell \leq k(n-1)$, there exists a path family from $S$ to $R'$ with degree $(k-1)(n-1)+\ell$. So $\deg(P_1) \equiv_n (k-1)(n-1)+\ell$. In order for $\deg(P_1) \leq k(n-1)$, we must have $\deg(P_1) = (k-1)(n-1)+\ell - j_1n$ for some nonnegative integer $j_1$. Similarly, $\deg(P_2)=(n-1)(m-k)-\ell-j_2n$ for some nonnegative integer $j_2$. Let \[\wt_{\Omega_k}(P) = \left(\prod_{i=0}^{n-1}x_1^{(i)}\right)^{j_1} \wt(P) \left(\prod_{i=0}^{n-1}x_m^{(i)}\right)^{j_2}.\]

\begin{theorem} \label{thm:comb_omega}
For $1 \leq k \leq m-1$,
\[\Omega_{k}^{(r)}(\x_1, \dots, \x_m) = \sum_{
P \in \cP^{(r)}_{\leq (m-1)(n-1)}} \wt_{\Omega_k}(P).\]
\end{theorem}

\begin{proof}
Recall that
\[\Omega_{k}^{(r)}(\x_1, \dots, \x_m) = \sum_{\ell=0}^{n-1}\sigma^{(r)}_{(n-1)(k-1)+\ell}(\x_1,\dots,\x_k)\bar{\sigma}^{(r+k-1-\ell)}_{(n-1)(m-k)-\ell}(\x_{k+1},\dots,\x_m).\] 
Consider the term 
\[\sigma^{(r)}_{(n-1)(k-1)+\ell}(\x_1,\dots,\x_k)
\bar{\sigma}^{(r+k-1-\ell)}_{(n-1)(m-k)-\ell}(\x_{k+1},\dots,\x_m)\] 
for some $0 \leq \ell \leq n-1$. 

By Theorem \ref{thm:comb_sigma}, $\sigma^{(r)}_{(n-1)(k-1)+\ell}(\x_1,\dots,\x_k)$ is the generating function for families of noncrossing highway paths starting from $S = [n] \setminus\{r+1\}$ and ending at $[n] \setminus \{r+k-1-\ell\}$ with degree at most $(n-1)(k-1)+\ell$, 
and $\bar{\sigma}^{(r+k-1-\ell)}_{(n-1)(m-k)-\ell}(\x_{k+1},\dots,\x_m)$ is the generating function for families of noncrossing highway paths starting at $[n] \setminus \{r+k-\ell\}$ and ending at $R = [n] \setminus \{r+m-1\}$ with degree at most $(n-1)(m-k)-\ell$. Thus, for each $\ell$, noncrossing path families corresponding to $\sigma^{(r)}_{(n-1)(k-1)+\ell}(\x_1,\dots,\x_k)$ and $\bar{\sigma}^{(r+k-1-\ell)}_{(n-1)(m-k)-\ell}(\x_{k+1},\dots,\x_m)$ connect between the $\x_k$ and $\x_{k+1}$ demarcation. Therefore, each term in $\Omega_k$ corresponds to some path family $P: S \to R$. 

Conversely, to define $\wt_{\Omega_k}(P)$, we have already shown that any path family $P: S \to R$ breaks up into $P_1 \in \cP^{(r)}_{\leq (k-1)(n-1)+\ell}$ and $P_2 \in \cP^{(r+k-1-\ell)}_{\leq (m-k)(n-1)-\ell}$ for a unique $0 \leq \ell \leq n-1$.

Lastly, one can check that \[\wt_{\Omega_k}(P) = \wt_{\sigma_{(n-1)(k-1)+\ell}}(P_1) \wt_{\bar{\sigma}_{(n-1)(m-k)-\ell}}(P_2).\]
\end{proof}
\begin{remark}
Since the sum is always over $P \in \cP_{\leq (m-1)(n-1)}^{(r)}$, this theorem implies that the number of terms in $\Omega_{k}^{(r)}(\x_1, \dots, \x_m)$ is constant for different $k$. When $k = 1$, $j_1$ must be zero for any path family, which implies that $\wt_{\Omega_1}(P) = \wt_{\bar{\sigma}_{(n-1)(m-1)}}(P)$. Similarly, $\wt_{\Omega_{m-1}}(P) = \wt_{\sigma_{(n-1)(m-1)}}(P)$.
\end{remark}
\begin{example} Consider $\Omega_k^{(3)}(\x_1, \x_2, \x_3,\x_4)$ where $n = 3$, which consists of monomials of length $(n-1)(m-1) = 6$. We calculate that $S = \{2,3\}$ and $R = \{1,2\}$. Two path families $P, Q: S \to R$ are depicted below. Since $\deg(P) = 6$, $\wt_{\Omega_k}(P) = \wt(P)$.
\begin{figure}
    \centering
\begin{tikzpicture}
\fill[gray!10!white] (0,0.25) rectangle (5,3.25);
\path
(0,1) edge (5,1)
(0,2) edge (5,2)
(0,3) edge (5,3)
(1,0.25) edge (1,3.25)
(2,0.25) edge (2,3.25)
(3,0.25) edge (3,3.25)
(4,0.25) edge (4,3.25);
\path[dashed]
(0,0.25) edge (5,0.25)
(0,3.25) edge (5,3.25);

\node at (1.3,0.7) {\scriptsize$x_1^{(3)}$};
\node at (2.3,0.7) {\scriptsize$x_2^{(1)}$};
\node at (3.3,0.7) {\scriptsize$x_3^{(2)}$};
\node at (4.3,0.7) {\scriptsize$x_4^{(3)}$};
\node at (1.3,1.7) {\scriptsize$x_1^{(2)}$};
\node at (2.3,1.7) {\scriptsize$x_2^{(3)}$};
\node at (3.3,1.7) {\scriptsize$x_3^{(1)}$};
\node at (4.3,1.7) {\scriptsize$x_4^{(2)}$};
\node at (1.3,2.7) {\scriptsize$x_1^{(1)}$};
\node at (2.3,2.7) {\scriptsize$x_2^{(2)}$};
\node at (3.3,2.7) {\scriptsize$x_3^{(3)}$};
\node at (4.3,2.7) {\scriptsize$x_4^{(1)}$};
\draw[->, >=stealth,YellowGreen,line width=0.8mm]
(0,1) -- (1,1) -- (2,1) -- (3,1) -- (4,1) -- (4,2) -- (5,2);
\draw[->, >=stealth,OliveGreen,line width=0.8mm]
(0,2) -- (1,2) -- (2,2) -- (3,2) -- (3,3) -- (4,3) -- (5,3);
\end{tikzpicture}
    \caption{$P:S \to R$ such that $\wt_{\Omega_k}(P) = \wt(P)$ for $k = 1,2,3$.}
    \label{fig:omega-example-2}
\end{figure}
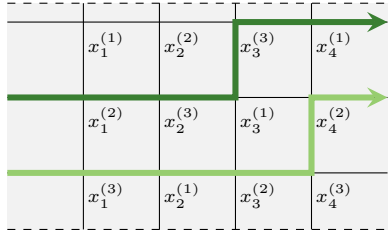 
On the other hand, the path family $Q$ has $\deg(Q) = 3$, so we expect that $j_1 + j_2 = 1$. We will explicitly calculate $j_1$ and $j_2$ when $k = 2$ and $\wt_{\Omega_2}(Q)$. When we cut between the second and the third vertical loop, the sink set of the path family restricted to the network on the left is $\{1,2\}$. Therefore $r + k-1 - \ell = 3$, which implies that $\ell = r+k-1-3 = 3+2-1-3 = 1$. Let the path families of the left and right networks be $Q_1$ and $Q_2$. Since $\deg(Q_1) = 3 = (n-1)(k-1)+\ell$ and $\deg(Q_2) = 0 = (n-1)(m-k)-\ell-3$, we have $j_1 = 0$ and $j_2 = 1$. Therefore, $\wt_{\Omega_2}(Q) = \wt(Q)x_4^{(3)}x_4^{(2)}x_4^{(1)} = x_1^{(3)}x_1^{(2)}x_2^{(1)} x_4^{(3)}x_4^{(2)}x_4^{(1)}$. 
\begin{figure}[H]
\centering
\begin{minipage}{0.3\textwidth}
\begin{tikzpicture}[scale=0.85]
\fill[gray!10!white] (0,0.25) rectangle (5,3.25);
\path
(0,1) edge (5,1)
(0,2) edge (5,2)
(0,3) edge (5,3)
(1,0.25) edge (1,3.25)
(2,0.25) edge (2,3.25)
(3,0.25) edge (3,3.25)
(4,0.25) edge (4,3.25);
\path[dashed]
(0,0.25) edge (5,0.25)
(0,3.25) edge (5,3.25);
\path[dashed, Blue,line width = 0.4mm]
(1.55,0) edge (1.55,3.5);
\node at (1.3,0.7) {\scriptsize$x_1^{(3)}$};
\node at (2.3,0.7) {\scriptsize$x_2^{(1)}$};
\node at (3.3,0.7) {\scriptsize$x_3^{(2)}$};
\node at (4.3,0.7) {\scriptsize$x_4^{(3)}$};
\node at (1.3,1.7) {\scriptsize$x_1^{(2)}$};
\node at (2.3,1.7) {\scriptsize$x_2^{(3)}$};
\node at (3.3,1.7) {\scriptsize$x_3^{(1)}$};
\node at (4.3,1.7) {\scriptsize$x_4^{(2)}$};
\node at (1.3,2.7) {\scriptsize$x_1^{(1)}$};
\node at (2.3,2.7) {\scriptsize$x_2^{(2)}$};
\node at (3.3,2.7) {\scriptsize$x_3^{(3)}$};
\node at (4.3,2.7) {\scriptsize$x_4^{(1)}$};
\draw[->, >=stealth,YellowGreen,line width=0.8mm]
(0,1) -- (1,1) -- (2,1) -- (3,1) -- (3,2) -- (4,2) -- (4,3) -- (5,3);
\draw[->, >=stealth,OliveGreen,line width=0.8mm]
(0,2) -- (1,2) -- (2,2) -- (2,3) -- (3,3) -- (3,3.25)
(3,0.25) -- (3,1) -- (4,1) -- (4,2) -- (5,2);
\node at (2.5,-0.5) {\scriptsize $k = 1, \ell = 2$, $j_1 = 0$, $j_2 = 1$};
\end{tikzpicture}
\end{minipage}
\begin{minipage}{0.3\textwidth}
\begin{tikzpicture}[scale=0.85]
\fill[gray!10!white] (0,0.25) rectangle (5,3.25);
\path
(0,1) edge (5,1)
(0,2) edge (5,2)
(0,3) edge (5,3)
(1,0.25) edge (1,3.25)
(2,0.25) edge (2,3.25)
(3,0.25) edge (3,3.25)
(4,0.25) edge (4,3.25);
\path[dashed]
(0,0.25) edge (5,0.25)
(0,3.25) edge (5,3.25);
\path[dashed, Blue,line width = 0.4mm]
(2.55,0) edge (2.55,3.5);
\node at (1.3,0.7) {\scriptsize$x_1^{(3)}$};
\node at (2.3,0.7) {\scriptsize$x_2^{(1)}$};
\node at (3.3,0.7) {\scriptsize$x_3^{(2)}$};
\node at (4.3,0.7) {\scriptsize$x_4^{(3)}$};
\node at (1.3,1.7) {\scriptsize$x_1^{(2)}$};
\node at (2.3,1.7) {\scriptsize$x_2^{(3)}$};
\node at (3.3,1.7) {\scriptsize$x_3^{(1)}$};
\node at (4.3,1.7) {\scriptsize$x_4^{(2)}$};
\node at (1.3,2.7) {\scriptsize$x_1^{(1)}$};
\node at (2.3,2.7) {\scriptsize$x_2^{(2)}$};
\node at (3.3,2.7) {\scriptsize$x_3^{(3)}$};
\node at (4.3,2.7) {\scriptsize$x_4^{(1)}$};
\draw[->, >=stealth,YellowGreen,line width=0.8mm]
(0,1) -- (1,1) -- (2,1) -- (3,1) -- (3,2) -- (4,2) -- (4,3) -- (5,3);
\draw[->, >=stealth,OliveGreen,line width=0.8mm]
(0,2) -- (1,2) -- (2,2) -- (2,3) -- (3,3) -- (3,3.25)
(3,0.25) -- (3,1) -- (4,1) -- (4,2) -- (5,2);
\node at (2.5,-0.5) {\scriptsize $k = 2, \ell = 1$, $j_1 = 0$, $j_2 = 1$};
\end{tikzpicture}
 \end{minipage}
 \begin{minipage}{0.3\textwidth}
\begin{tikzpicture}[scale=0.85]
\fill[gray!10!white] (0,0.25) rectangle (5,3.25);
\path
(0,1) edge (5,1)
(0,2) edge (5,2)
(0,3) edge (5,3)
(1,0.25) edge (1,3.25)
(2,0.25) edge (2,3.25)
(3,0.25) edge (3,3.25)
(4,0.25) edge (4,3.25);
\path[dashed]
(0,0.25) edge (5,0.25)
(0,3.25) edge (5,3.25);
\path[dashed, Blue,line width = 0.4mm]
(3.55,0) edge (3.55,3.5);
\node at (1.3,0.7) {\scriptsize$x_1^{(3)}$};
\node at (2.3,0.7) {\scriptsize$x_2^{(1)}$};
\node at (3.3,0.7) {\scriptsize$x_3^{(2)}$};
\node at (4.3,0.7) {\scriptsize$x_4^{(3)}$};
\node at (1.3,1.7) {\scriptsize$x_1^{(2)}$};
\node at (2.3,1.7) {\scriptsize$x_2^{(3)}$};
\node at (3.3,1.7) {\scriptsize$x_3^{(1)}$};
\node at (4.3,1.7) {\scriptsize$x_4^{(2)}$};
\node at (1.3,2.7) {\scriptsize$x_1^{(1)}$};
\node at (2.3,2.7) {\scriptsize$x_2^{(2)}$};
\node at (3.3,2.7) {\scriptsize$x_3^{(3)}$};
\node at (4.3,2.7) {\scriptsize$x_4^{(1)}$};
\draw[->, >=stealth,YellowGreen,line width=0.8mm]
(0,1) -- (1,1) -- (2,1) -- (3,1) -- (3,2) -- (4,2) -- (4,3) -- (5,3);
\draw[->, >=stealth,OliveGreen,line width=0.8mm]
(0,2) -- (1,2) -- (2,2) -- (2,3) -- (3,3) -- (3,3.25)
(3,0.25) -- (3,1) -- (4,1) -- (4,2) -- (5,2);
\node at (2.5,-0.5) {\scriptsize $k = 3, \ell = 2$, $j_1 = 1$, $j_2 = 0$};
\end{tikzpicture}
 \end{minipage}
    \caption{A path family $Q: S \to R$ such that $\wt_{\Omega_1}(Q) = \wt_{\Omega_2}(Q) \neq \wt_{\Omega_3}(Q)$}
    \label{fig:omega-example-1}
\end{figure}
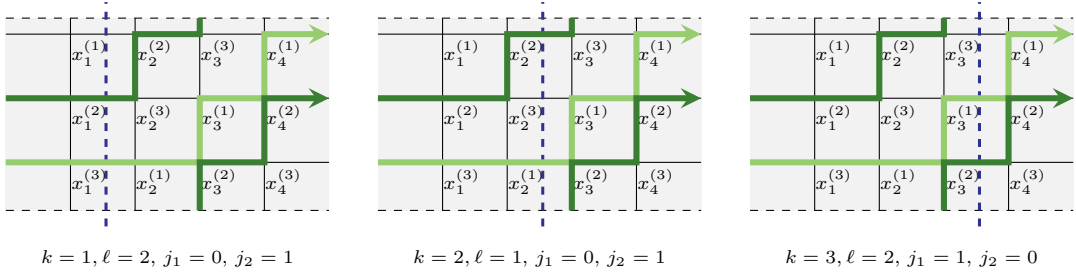
\end{example}

\section{Questions}
We conclude with some questions for future work.

\vspace{0.1in}
\noindent\textbf{Question 1.}  What are the explicit formulas for general permutations?
\vspace{0.1in}

While the results of the present paper are limited to special permutations, we explored some other permutations and our finding suggested that there may be nice formulas in general. An interesting example is $s(\x_2^{(1)})$, where $n = 2$, $m = 4$, and $s=s_2s_3s_1s_2 = (13)(24)$. The following factor appears in the numerator of $s(\x_2^{(1)})$:
\begin{align*}
    & x_1^{(1)} x_1^{(2)} x_2^{(1)} x_2^{(2)} 
    + x_1^{(1)} x_1^{(2)} x_2^{(1)} x_3^{(2)} 
    + x_1^{(2)} x_2^{(1)} x_2^{(1)} x_3^{(2)} 
    + x_1^{(2)} x_2^{(1)} x_3^{(1)} x_3^{(2)} 
    + x_1^{(1)} x_1^{(2)} x_2^{(1)} x_4^{(2)} 
    + x_1^{(2)} x_2^{(1)} x_2^{(1)} x_4^{(2)} \\
    + \ & x_1^{(1)} x_1^{(2)} x_3^{(1)} x_4^{(2)} 
    + 2 x_1^{(2)} x_2^{(1)} x_3^{(1)} x_4^{(2)} 
    + x_2^{(1)} x_2^{(2)} x_3^{(1)} x_4^{(2)} 
    + x_1^{(2)} x_3^{(1)} x_3^{(1)} x_4^{(2)} 
    + x_2^{(2)} x_3^{(1)} x_3^{(1)} x_4^{(2)} 
    + x_1^{(2)} x_2^{(1)} x_4^{(1)} x_4^{(2)} \\
    + \ & x_1^{(2)} x_3^{(1)} x_4^{(1)} x_4^{(2)} 
    + x_2^{(2)} x_3^{(1)} x_4^{(1)} x_4^{(2)} 
    + x_3^{(1)} x_3^{(2)} x_4^{(1)} x_4^{(2)}.
\end{align*}

Unlike our $\Omega$ functions, some monomials in this factor have a coefficient of $2$ or contain squares, such as $(x_2^{(1)})^2$ and $(x_3^{(1)})^2$. It would be interesting to understand this factor as an example of a generalization of our $\Omega$ functions and interpret it in terms of cylindric networks.

\vspace{0.1in}
\noindent\textbf{Question 2.}  Is there a combinatorial proof of algebraic identities such as \lemref{thm:identity} and \lemref{lem:identity} using cylindric networks?
\vspace{0.1in}

Currently, our proofs of \lemref{thm:identity} and \lemref{lem:identity} rely on only elementary algebra. As we have combinatorially interpreted factors involved in these identities, it is natural to look for combinatorial proofs. One possibility would be that the two sides of a desired identity are two different ways of writing the sum of weights of a certain set of path families.

\vspace{0.1in}
\noindent\textbf{Question 3.}  Can cluster algebraic methods be used to prove our formulas?
\vspace{0.1in}

In~\cite{ILP2016}, Inoue, Lam, and Pylyavskyy define the \emph{cluster $R$-matrix}, a transformation obtained from a sequence of cluster mutations.  They then use a change of variables to obtain the birational $R$-matrix from the cluster $R$-matrix.  Motivated by these results, the connection between cluster algebras and the plabic $R$-matrix, a generalization of the birational $R$-matrix, was further studied by the first author in~\cite{C2020}.  These connections cluster algebras may be able to be exploited to find more elegant proofs of our formulas or to extend our results.

\section*{Acknowledgements}
This research was partially conducted at the 2020 University of Minnesota Twin Cities REU, which was supported by NSF RTG grant DMS-1745638. We thank Pavlo Pylyavskyy for suggesting this problem and Emily Tibor for her support and her feedback on this manuscript and various presentations. 

\clearpage
\bibliographystyle{alpha}
\bibliography{references.bib}

\end{document}